\documentclass[12pt,twoside]{article}

\usepackage[T1]{fontenc}
\usepackage[latin1]{inputenc}
\usepackage[english]{babel}
\usepackage[babel]{csquotes}

\usepackage{cite}

\usepackage{amssymb}
\usepackage{amsmath}
\usepackage{amsthm}
\usepackage{latexsym}
\usepackage{graphicx}
\usepackage{mathrsfs}
\usepackage{mathrsfs}
\usepackage{bbm}

\usepackage[usenames,dvipsnames]{color}

\theoremstyle{plain}
\newtheorem{thm}{Theorem}[section]

\newtheorem{lem}[thm]{Lemma}

\newtheorem{defin}[thm]{Definition}
\theoremstyle{definition}
\newtheorem{rmk}[thm]{Remark}

\def\enne{\mathbb{N}}
\def\zeta{\mathbb{Z}}

\def\erre{\mathbb{R}}

\def\P{\mathbb{P}}

\def\E{\mathop{{}\mathbb{E}}}
\def\cL{\mathscr{L}}
\def\cF{\mathscr{F}}
\def\cB{\mathscr{B}}
\def\eps{\varepsilon}

\def\beq{\begin{equation}}
\def\eeq{\end{equation}}

\def\to{\rightarrow}
\def\wto{\rightharpoonup}
\def\wstarto{\stackrel{*}{\rightharpoonup}}
\def\embed{\hookrightarrow}
\def\cembed{\stackrel{c}{\hookrightarrow}}

\def\norm #1{\left\|#1\right\|}
\def\abs #1{\left|#1\right|}
\def\sp #1#2{\left<#1,#2\right>}

\title{\huge\rm On the stochastic Cahn-Hilliard equation\\
with a singular double-well potential\footnote{{\bf Acknowledgments.}
		The author is very grateful to Carlo Marinelli for his expert support and 
		fundamental advice which led to a better presentation of these results.}
		\\[.5cm]}

\author{{\large\sc Luca Scarpa}\\
		{\small Department of Mathematics, University College London}\\
		{\small Gower Street, London WC1E 6BT, United Kingdom}\\
		{\small E-mail: \texttt{luca.scarpa.15@ucl.ac.uk}}
		}

\date{}

\begin{document}

\maketitle  

\begin{abstract}
  We prove well-posedness and regularity
  for the stochastic pure Cahn-Hilliard equation
  under homogeneous Neumann boundary conditions, with both additive and multiplicative Wiener noise. 
  In contrast with great part of the literature, the double-well potential is treated as generally 
  as possible, its convex part being associated to 
  a multivalued maximal monotone graph everywhere defined 
  on the real line on which no growth nor smoothness assumptions are assumed.
  The regularity result allows to give appropriate sense
  to the chemical potential and to write a natural variational formulation of the problem.
  The proofs are based on suitable monotonicity and compactness arguments
  in a generalized variational framework.
  \\[.5cm]
\end{abstract}

\pagestyle{myheadings}
\newcommand\testopari{\sc Luca Scarpa}
\newcommand\testodispari{\sc On the stochastic Cahn-Hilliard equation with a singular potential}
\markboth{\testodispari}{\testopari}


\thispagestyle{empty}

\section{Introduction}
\setcounter{equation}{0}
\label{intro}
The well-known Cahn-Hilliard equation was first introduced in \cite{cahn-hill} to describe the evolution 
of the phase separation phenomenon involving a binary metallic alloy:
in general, in a solid-solid phase separation each phase concentrates,
and this results in what is usually referred to as {\em spinodal decomposition}.

The classical Cahn-Hilliard equation
reads
\[
  \partial_t u - \Delta w =0\,, \quad w\in-\Delta u + \beta(u) + \pi(u) + g \qquad\text{in } (0,T)\times D\,,
\]
where $D\subseteq\erre^N$ ($N=2,3$) is a smooth bounded domain with smooth boundary $\Gamma$,
$T>0$ is a fixed finite time, $\Delta$ stands for the Laplacian acting on the space variables and 
$g$ is a given source.
The unknown
$u$ and $w$ represent the order parameter and the chemical potential, respectively. 
Here, $\beta$ is the subdifferential of the convex part $j$ and $\pi$ is the derivative of 
the concave perturbation $\widehat\pi$ of a so-called double-well potential $\psi:=j+\widehat\pi$.
Typical examples of $\psi$ (see also \cite{col-gil-spr}) are given by
\begin{gather*}
\psi_{reg}(r)=\frac14(r^2-1)^2\,, \quad r\in\erre\,,\\
\psi_{log}(r)=\left((1+r)\ln(1+r) - (1-r)\ln(1-r)\right)-cr^2\,, \quad r\in(-1,1)\,, \quad c>0\,,\\
\psi_{2obst}(r)=\begin{cases}
  c(1-r^2) \quad&\text{if } |r|\leq1\,,\\
  +\infty \quad&\text{if } |r|>1\,,
\end{cases}
\end{gather*}
which correspond to a regular, logarithmic and non-smooth double-well potential, respectively
(the last one is usually considered in the so-called double-obstacle problem).
In the simplest case, the equation is coupled with homogeneous Neumann boundary conditions
for both $u$ and $w$,
and a given initial datum:
\[
  \partial_{\bf n}u = \partial_{\bf n}w = 0 \quad\text{on } (0,T)\times\Gamma\,,
  \qquad u(0)=u_0 \quad\text{in } D\,,
\]
where the symbol ${\bf n}$ stands for the outward normal unit vector on $\Gamma$.
It is well-known that the homogeneous Neumann condition for the chemical potential 
ensures the conservation of the mean-value of $u$ on $D$, as it is easily proved
integrating the first equation on $D$.
It is also noteworthy that the term $w-g$ associated to the chemical potential
results from the subdifferentiation of the Ginzburg-Landau
free energy functional
\[
  \mathcal E(x):=\frac12\int_D|\nabla x|^2 + \int_D\left(j(x)+\widehat\pi(x)\right)\,.
\]

From a mathematical perspective, deterministic Cahn-Hilliard equations have 
received much attention in the last years and have been analytically 
investigated also in more general frameworks, such as the viscous case and
under the so-called dynamic boundary conditions.
A special mention goes to the contributions 
\cite{col-fuk-eqCH, col-gil-spr, cal-colli, cher-gat-mir, cher-mir-zel, cher-pet, colli-fuk-CHmass,
gil-mir-sch} dealing with 
global well-posedness and regularity for Cahn-Hilliard and Allen-Cahn type
equations with singular potentials, and \cite{col-fuk-diffusion, col-scar, gil-mir-sch-longtime}
regarding
asymptotics and long-time behaviour of solutions. Also, we point out the papers
\cite{col-far-hass-gil-spr, col-gil-spr-contr, col-gil-spr-contr2, hinter-weg} concerning optimal 
control problems related to Cahn-Hilliard systems.

While the deterministic Cahn-Hilliard equation provides a good description of the spinodal
decomposition process, on the other hand it it is not effective in taking into account
the effects due to the random solute vibrational movements.
These can be accounted for directly by adding a cylindrical Wiener process $W$
in the equation itself, hence getting a stochastic partial differential equation of the form
\[
  du(t) - \Delta w(t)\,dt = B(t,u(t))\,dW_t\,, \quad w\in-\Delta u + \beta(u) + \pi(u) + g
  \qquad\text{in } (0,T)\times D\,,
\]
with homogeneous Neumann
conditions for $u$ and $w$, and a given initial value $u_0$,
where $B$ is a suitable stochastically integrable operator.

The available literature on
stochastic Cahn-Hilliard equations is not as extended as the corresponding deterministic one
and is mainly focused
on the classical case of a smooth polynomial double-well potential.
Let us point out the contribution \cite{daprato-deb}, in which the authors prove 
existence and regularity of solutions, as well as existence and uniqueness of
an invariant measure for the transition semigroup, in the case of a polynomial double-well 
potential of even degree $2p$. Moreover,
in the case of the regular double-well potential $\psi_{reg}$,
existence, uniqueness and
regularity of weak statistical solution and existence of a strong solution are proved
in \cite{elez-mike} for the equation in local form and in \cite{corn} for a nonlocal version.
Finally, let us also mention the contributions \cite{ant-kar-mill} on a stochastic Cahn-Hilliard 
equation with unbounded noise and \cite{deb-zamb, deb-goud, goud} dealing
with stochastic Cahn-Hilliard equations with reflections.

The noteworthy feature of this work is that neither growth nor smoothness assumptions on $\beta$
are required, provided that $\beta$ is everywhere defined: consequently, 
in contrast with great part of the existing literature, we are able to handle 
any double-well potential $\psi:\erre\to\erre$, not necessarily smooth, with any arbitrary
order of growth
(e.g.~also super-exponential).
Note that the requirement $D(\beta)=\erre$ seems
to be fundamental in the stochastic setting, but is not necessary in the deterministic case.
The second
evident difficulty in the stochastic case is that the conservation of the mean value
of the order parameter $u$ is no longer true, unless the covariance operator $B$ takes values in 
the subspace of null-mean elements of $L^2(D)$: although this hypothesis is usually required in the literature,
we show that in the case of additive noise it can be avoided, provided that a stronger condition on the 
moments of $B$ and $u_0$ holds.
The high singularity of $\beta$
and the lack of growth assumptions prevent us to rely on the classical
semigroup or variational approaches.
For this reason, we establish global well-posedness and regularity
in a suitable generalized variational setting:
appropriate regularized equations are solved in the classical variational framework, then
uniform estimates are found and a passage to the limit
provides solutions to the original problem. These arguments 
have been also employed in \cite{mar-scar-diss, mar-scar-div, mar-scar-div2, scar-div,
mar-scar-ref, mar-scar-erg, mar-scar-note}
for dissipative and divergence-form SPDEs with singular drift, and in \cite{orr-scar}
for singular stochastic Allen-Cahn equations with dynamic boundary conditions.
The main idea is to
rely on $L^1$-estimates and weak compactness criterions in $L^1$-spaces, which were first 
employed in the works \cite{barbu-semilin, barbu-daprato-rock} for semilinear and porous-media 
equations.

Let us briefly present the structure of the paper.
We consider the stochastic Cahn-Hilliard system with homogeneous Neumann boundary conditions in the form
\begin{align}
  \label{eq1}
  du(t) - \Delta w(t)\,dt = B(t, u(t))\,dW_t \qquad&\text{in } (0,T)\times D\,,\\
  \label{eq2}
  w \in -\Delta u + \beta(u) + \pi(u) + g \qquad&\text{in } (0,T)\times D\\
  \label{eq3}
  \partial_{\bf n} u=0\,, \quad \partial_{\bf n}w=0 \qquad&\text{in } (0,T)\times\Gamma\,,\\
  \label{eq4}
  u(0)=u_0 \qquad&\text{in } D\,,
\end{align}
where $W$ is a cylindrical Wiener process on a certain Hilbert space $U$, $B$ takes values in the space
of Hilbert-Schmidt operators from $U$ to $L^2(D)$, $\beta$ is a maximal monotone graph everywhere defined on $\erre$,
$\pi$ is a Lipschitz function and $g$ and $u_0$ are given data.

In Section~\ref{main_results}, we fix the main assumptions that will be in order throughout the work, 
we precise the concept of strong solution that will be used and we state the main results of the paper.

In Section~\ref{additive}, we prove well-posedness for the problem with additive noise.
As we have anticipated, the proof of existence is carried out considering suitable regularized equations, in which $\beta$
is replaced by its Yosida approximation $\beta_\lambda$ and the operator $B$ is 
smoothed out through a suitable power of the resolvent of the laplacian $(I-\eps\Delta)^{-k}$.
Uniform estimates (both pathwise and in expectation) 
on the solutions to the regularized equation are proved and global solutions to the problem are obtained by
passing to the limit as $\lambda, \eps \searrow0$
in suitable topologies through compactness and monotonicity arguments. 
More specifically, we pass to the limit pathwise on each trajectory as $\lambda\searrow0$, 
with $\eps>0$ being fixed, and then we let $\eps\searrow0$ 
using suitable convergences in expectation, removing thus the regularization on the noise.
The continuous dependence property follows directly from a generalized It\^o's formula
and monotonicity.

Section~\ref{mult} contains the proof of well-posedness with multiplicative noise.
This is performed in a classical way, using a fixed-point argument on sufficiently small
subintervals of $[0,T]$ and by a standard patching technique.

Finally, in Section~\ref{reg} we show that if the data satisfy additional assumptions, then 
the solutions to the problem with additive noise inherit a further regularity.
In particular, the chemical potential is found to be $H^1$-valued, so that 
a natural variational formulation of the problem can be written.
The proof consists in showing a further class of uniform estimates on the approximated solutions
involving a regularized version of the Ginzburg-Landau free-energy functional.


\section{Assumptions and main results}
\setcounter{equation}{0}
\label{main_results}

\subsection{Notation, setting and assumptions}
Throughout the paper, $(\Omega,\cF, \P)$ is a probability space endowed with a
filtration $(\cF_t)_{t\in[0,T]}$ which is saturated and right-continuous, where $T>0$ is a fixed final time.
Moreover,  $D\subseteq\erre^N$ is a smooth bounded domain with smooth boundary $\Gamma$
($N=2,3$).

For any Banach space $E$, we use the symbols $L^p(0,T; E)$ and $L^p(\Omega; E)$ to indicate the classes
of strongly measurable $p$-Bochner integrable functions $(0,T)\to E$ and $\Omega\to E$, respectively. We 
recall that if $E$ is separable then strong and weak measurability coincide and we will drop the qualifier 
"strong" in such a case. We say that a $E$-valued process
is measurable if it is jointly measurable from $\cF\otimes\cB([0,T])$ to $\cB(E)$.
For any two Hilbert spaces $U_1$ and $U_2$, we use the symbols $\cL(U_1, U_2)$, $\cL_1(U_1, U_2)$ and
$\cL_2(U_1, U_2)$ for the spaces of linear, trace class and Hilbert-Schmidt operators from $U_1$ to $U_2$,
respectively.

We introduce the spaces
\[
  H:=L^2(D)\,, \qquad V_s:=
  \begin{cases}
  H^s(D) \qquad&\text{if } s\in[1,2)\,,\\
  \left\{v \in H^s(D): \;\partial_{\bf n}v=0 \;\text{ on } \Gamma\right\}\quad&\text{if } s\geq 2\,.
  \end{cases}
\]
Duality pairings, scalar products and norms are denoted by the symbols
$\sp\cdot\cdot$, 
$(\cdot,\cdot)$ and $\norm{\cdot}$, respectively, with a subscript specifying the spaces in consideration.
Let us point out that thanks to the classical results on elliptic regularity, 
a possible norm on $V_2$, equivalent to the usual one, is
\[
  \norm{v}_2:=\sqrt{\norm{v}_H^2 + \norm{\Delta v}_H^2}\,, \quad v\in V_2\,.
\]
We shall use the notation
\[
  y_D:=\frac1{|D|}\sp{y}1_{V_1} \qquad\forall\,y\in V_1^*
\]
for the mean operator in $V_1^*$ and recall that a possible norm on $V_1$, equivalent to the usual one, is given by
\[
  \norm{v}_1:=\sqrt{|v_D|^2 + \norm{\nabla v}^2_H}\,, \qquad v\in V_1\,.
\]
We introduce
\[
  H_0:=\left\{v \in H:\; v_D=0\right\}\,,
\]
which is a Hilbert space with the scalar product of $H$. 

\begin{rmk}
Let us point out that in the sequel we will use the following classical result:
for any $x\in V_1$ and $f:\erre\to\erre$ Lipschitz continuous, then $f(x)\in V_1$ and
\[
  \nabla f(x) = f'(x)\nabla x \quad\text{a.e.~in } D\,.
\]
Since $f'$ is defined only almost everywhere on $\erre$, the relation above 
could make no sense (e.g.~if $x$ takes values, with positive measure on $D$, 
in a set on which $f'$ is not well-defined). It has to be 
implicitly intended that the quantities above are $0$ in such a case.
\end{rmk}

It is natural when dealing with Cahn-Hilliard equations to define the operator $\mathcal N$ as
\[
  \mathcal N:D(\mathcal N)\to V_1\,, \qquad D(\mathcal N):=\left\{y\in V_1^*:\; y_D=0\right\}\,,
\]
where $\mathcal Ny$ is the unique solution with null mean to the generalized Neumann problem,
\[
  \int_D\nabla\mathcal Ny \cdot\nabla\varphi=\sp{y}{\varphi}_{V_1} \quad\forall\,\varphi\in V_1\,, \qquad (\mathcal Ny)_D=0\,.
\]
Let us recall some important properties of $\mathcal N$ that will be used in the sequel: the reader can refer here
to \cite[pp.~979-980]{col-gil-spr}. First of all, $\mathcal N$ is an 
isomorphism between $D(\mathcal N)$ and $\{v\in V_1: v_D=0\}$ satisfying
\beq
  \label{N1}
  \sp{y_1}{\mathcal Ny_2}_{V_1}=\sp{y_2}{\mathcal Ny_1}_{V_1}=\int_D\nabla\mathcal Ny_1\cdot\nabla\mathcal Ny_2 
  \qquad\forall\,y_1,y_2\in D(\mathcal N)\,.
\eeq
Secondly, the operator
\[
\norm{\cdot}_*:V_1^*\to[0,+\infty)\,, \qquad \norm y_*^2:=\norm{\nabla\mathcal N(y-y_D)}_H^2 + |y_D|^2\,, \quad y\in V_1^*
\]
defines a norm on $V_1^*$, equivalent to the usual dual norm, such that:
\beq
  \label{N2}
  \norm v_H^2 \leq \eps\norm{\nabla v}_H^2 + C_\eps\norm v_*^2 \qquad\forall\,v\in V_1\,,
\eeq
for every $\eps>0$ and a positive constant $C_\eps$.
Finally, we also have that
\beq
  \label{N3}
  \sp{\partial_t v(t)}{\mathcal Nv(t)}_{V_1}=\frac12\frac{d}{dt}\norm{\nabla\mathcal Nv(t)}_H^2 \quad\text{for a.e.~}t\in(0,T)
\eeq
for every $v\in H^1(0,T; V_1^*)$ such that $v_D=0$ almost everywhere in $(0,T)$.

Let us fix now the main assumptions of the work.

First of all, we assume that $\beta:\erre\to2^{\erre}$ is a maximal monotone operator with effective domain
$D(\beta)=\erre$, $j:\erre\to[0,+\infty)$ is the proper, convex and lower semicontinuous function
such that $j(0)=0$ and $j^*$ is its convex conjugate. We make a symmetry hypothesis on $j$ of the type
\[
\limsup_{|x|\to+\infty}\frac{j(x)}{j(-x)}<+\infty\,,
\]
which is automatically satisfied if $j$ is even for example. Furthermore, $\pi:\erre\to\erre$ is a Lipschitz-continuous
function with Lipschitz constant $C_\pi>0$, and we set $C_0:=|\pi(0)|$.\\
The Yosida approximation of $\beta$ and the Moreau regularization of $j$ are defined as
\[
  \beta_\lambda:=\frac{I-(I+\lambda\beta)^{-1}}{\lambda}\,, \qquad
  j_\lambda(x):=\inf_{y\in\erre}\left\{\frac{|x-y|^2}{2\lambda}+j(y)\right\}\,, \quad x\in\erre\,,
\]
for every $\lambda>0$: see \cite[Chapter~2]{barbu-monot} for further properties.

Secondly, we assume that $g$ and the initial datum $u_0$ satisfy
\begin{gather}
  \label{g}
  g\in L^2(\Omega; L^2(0,T; H)) \qquad\text{progressively measurable}\,\\
  \label{u0}
  u_0 \in L^2(\Omega,\cF_0,\P; H)\,, \qquad j(\alpha(u_0)_D) \in L^1(\Omega) \quad\forall\,\alpha>0\,.
\end{gather}

Finally, as far as the noise is concerned, 
we suppose that $U$ is a Hilbert space and $W$ is a cylindrical Wiener process on $U$. 
The hypotheses on $B$ are slightly different
depending on whether we are considering additive or multiplicative noise, and will be recalled explicitly 
in the main results of the work.
In case of additive noise, we assume
\begin{gather}
\label{add1}
  B \in L^2(\Omega; L^2(0,T; \cL_2(U,H)))\quad\text{progressively measurable}\,,\\
\label{add2}
  j(\alpha(B\cdot W)_D) \in L^1(\Omega\times(0,T)) \qquad\forall\,\alpha>0\,.
\end{gather}
In case of multiplicative noise, we assume
\beq\label{mult1}
  B:\Omega\times[0,T]\times V_1\to\cL_2(U,H_0) \quad\text{progressively measurable}
\eeq
and that there exist a constant $C_B>0$ and a process $f\in L^2(\Omega\times(0,T))$
such that
\begin{align}
  \label{mult2}
  \norm{B(\omega,t,x_1)-B(\omega, t, x_2)}_{\cL_2(U,V_1^*)} \leq C_B\norm{x_1-x_2}_{V_1^*}& \qquad\forall\, x_1,x_2 \in V_1\,,\\
  \label{mult3}
  \norm{B(\omega,t,x)}_{\cL_2(U,H)}\leq |f(\omega,t)| + C_B\norm{x}_{V_1}& \qquad\forall\,x\in H
\end{align}
for every $(\omega,t)\in \Omega\times[0,T]$.
\begin{rmk}
Let us focus on the additive noise case and
comment more specifically on hypotheses \eqref{u0} and \eqref{add2}, giving some
sufficient conditions on $B$ and $u_0$ for these ones to hold.
First of all, \eqref{u0} is the the classical integrability assumption on the initial datum:
the requirement involving the parameter $\alpha$ is essential when dealing with 
singular potentials (e.g.~when $j$ grows faster than a polynomial). It is not difficult to check that 
this is satisfied when $(u_0)_D$ is bounded $\P$-almost surely, or when 
$j(\alpha u_0)\in L^1(\Omega\times D)$ for every $\alpha>0$, for example.
Secondly, \eqref{add2} is an existence assumption on certain moments of $B$
associated to the potential $j$: again, when $j$ is singular, 
the introduction of the parameter $\alpha$ is essential.
If $B$ takes values in $\cL_2(U,H_0)$, then $(B\cdot W)_D=0$ and
condition \eqref{add2} is automatically satisfied, for example.
Otherwise, if $j$ has polynomial growth, 
for example if $j$ is bounded from above by a polynomial of order $p\geq2$, then,
by the Jensen and Burkholder-Davis-Gundy inequalities, 
it is not difficult to check that a sufficient condition for \eqref{u0} and \eqref{add2} to hold is that
\[
  u_0\in L^p(\Omega; L^p(D))\,, \qquad B \in L^p(\Omega; L^2(0,T; \cL_2(U,H)))\,,
\]
which is essentially a stronger existence condition on the moments of $u_0$ and $B$. 
\end{rmk}
\begin{rmk}\label{rmk_additive}
  Let us compare the assumptions on $B$ in case of additive and multiplicative noise. 
  First of all, note that the fact that $B$ takes values in $\cL_2(U,H_0)$
  clearly implies condition \eqref{add2}, so that the hypotheses in case of multiplicative noise are more stringent.
  Secondly, in case 
  of multiplicative noise, the hypothesis that $B$ takes values in 
  the space $\cL_2(U,H_0)$ implies that the stochastic integral $B\cdot W$ has null mean at any time, and
  hence ensures the conservation of the mean of the solution $u$ to the system: this is a quite natural 
  assumption for stochastic Cahn-Hilliard systems and has been widely employed in the common literature
  (see e.g.~\cite{daprato-deb, elez-mike, goud}).
  However, it is very interesting to
  note that in case of additive noise, the conservation of the mean on $D$ of the stochastic component is not needed
  in order to have existence of solutions for the equation, provided at least that $B$ is regular enough so that
  \eqref{add2} is satisfied.
\end{rmk}
\begin{rmk}
  Note that the assumptions on $u_0$ and $B$ can be reformulated using the notation of 
  Orlicz spaces. In particular, if for any measure space $(E,\mathscr{E},\mu)$
  $M_j(E)$ and $L_j(E)$ denote the strong and weak Orlicz spaces associated to $j$ on $E$, respectively, i.e.
  \begin{gather*}
    M_j(E):=\left\{u:E\to\erre\quad \mathscr E\text{-measurable}: \quad j(\alpha u)\in L^1(E, \mu) \quad
    \forall\,\alpha>0\right\}\,,\\
    L_j(E):=\left\{u:E\to\erre \quad\mathscr E\text{-measurable}:\quad
    \exists\,\alpha>0 :\quad j(\alpha u)\in L^1(E,\mu)\right\}\,,
  \end{gather*}
  then conditions \eqref{u0} and \eqref{add2} can be reformulated as
  \[
    (u_0)_D \in M_j(\Omega)\,, \qquad (B\cdot W)_D \in M_j(\Omega\times(0,T))\,.
  \]
  If $j$ is a polynomial, then
  $M_j=L_j$, and 
   dependence on $\alpha$ can be passed-by as we have already seen in Remark~\ref{rmk_additive}.
  The relevance of Orlicz spaces is not surprising in this context: indeed, they play an important role in the study of 
  stochastic evolution equations from a variational approach (for further detail see 
  \cite[\S~4.3--4.4]{barbu-daprato-rock-book}).
\end{rmk}

\subsection{Concept of solution and main results}
We now give the definition of solution for problem \eqref{eq1}--\eqref{eq4} that will be used throughout the paper
and we state the main results.

First of all, we need to identify a possible weak variational formulation for the system \eqref{eq1}--\eqref{eq4}: the idea is to 
(formally) substitute equation \eqref{eq2} into \eqref{eq1} in order to get an evolution equation of monotone type just in
terms of the variable $u$. 

Let us proceed in a formal way for the moment. Assume that $(u,w)$ is a reasonable solution to our problem,
which means (in an appropriate sense) that
\[
  du(t) - \Delta w\,dt = B(t,u(t))\,dW_t\,, \qquad w=-\Delta u + \xi + \pi(u)+ g\,,
\]
where $\xi \in \beta(u)$ almost everywhere, $\partial_{\bf n}u=\partial_{\bf n}w=0$ and $u(0)=u_0$.
Substituting the second equation in the first one we can write the system only in terms of $u$ as
\[
  du(t) - \Delta\left(-\Delta u + \xi + \pi(u)+ g\right)\,dt = B(t,u(t))\,dW_t\,.
\]
Let us focus now on the operator acting on $u$ in this last equation. Note that for every test function $\varphi \in V_2$,
recalling that
$\partial_{\bf n}(-\Delta u + \xi + \pi(u)+ g)=0$ by
 the boundary condition for $w$
and integrating by parts, we have (again formally)
\[\begin{split}
  \int_D- \Delta\left(-\Delta u + \xi + \pi(u)+ g\right)\varphi&=
  -\int_D\left(-\Delta u + \xi + \pi(u)+ g\right)\Delta\varphi\\
  &=\int_D\Delta u\Delta\varphi - \int_D\xi\Delta\varphi - \int_D\pi(u)\Delta\varphi - \int_Dg\Delta\varphi\,.
\end{split}\]

Bearing in mind these formal considerations, we can now proceed in a rigorous way. First of all, we introduce the operators
\begin{gather*}
  -\Delta: V_1\to V_1^*\,, \qquad \sp{-\Delta x}\varphi_{V_1}:=\int_D\nabla x\cdot\nabla\varphi\,, \quad x,\varphi\in V_1\,,\\
  \Delta^2:V_2\to V_2^*\,, \qquad \sp{\Delta^2 x}\varphi_{V_2}:=\int_D\Delta x\Delta\varphi\,, \quad x,\varphi\in V_2\,.
\end{gather*}
Secondly, we use the same notation $-\Delta$ to indicate the natural extension of the operator $-\Delta$
defined above acting
on the whole space $H$, i.e.
\[
  -\Delta: H \to V_2^*\,, \qquad \sp{-\Delta x}\varphi_{V_2}:=-\int_Dx\Delta\varphi\,, \quad x\in H\,, \quad \varphi\in V_2\,.
\]
Moreover,
since we cannot expect $\xi$ to be an $H$-valued process, but just $L^1(D)$-valued
(this is quite common for stochastic equations with singular drift, as you can see in the works 
\cite{mar-scar-diss, mar-scar-div, mar-scar-div2, scar-div, orr-scar}),
we need to extend the operator $-\Delta$ to $L^1(D)$. To this end,
we define the Banach space
\[
\bar V_2:=\{x\in V_2: \Delta x\in L^\infty(D)\}
\]
with its natural norm
and note that the operator $-\Delta$
can be extended again (with the same notation) to
\[
  -\Delta: L^1(D)\to \bar V_2^*\,.
\]
With this notation, the formal computations that we have made above can be written as
\[
\int_D- \Delta\left(-\Delta u + \xi + \pi(u)+ g\right)\varphi=
\sp{\Delta^2 u}\varphi_{V_2} + \sp{-\Delta(\xi+\pi(u) + g)}\varphi_{\bar V_2} \qquad\forall\,\varphi\in \bar V_2\,.
\]

We are now ready to give the definition of strong solution to the problem.
\begin{defin}[Strong solution]
  A strong solution to the problem \eqref{eq1}--\eqref{eq4} is a pair $(u,\xi)$, where 
  $u$ is an $H$-valued predictable process, $\xi$ is an $L^1(D)$-valued predictable process, and such that
  \begin{gather*}
    u \in L^2(\Omega; C^0([0,T]; V_1^*))\cap L^2(\Omega; L^\infty(0,T; H))\cap L^2(\Omega; L^2(0,T; V_2))\,,\\
    \xi \in L^1(\Omega; L^1(0,T; L^1(D)))\,,\\
    j(u) + j^*(\xi) \in L^1(\Omega\times(0,T)\times D)\,,\\
    \xi \in \beta(u) \quad\text{a.e.~in } \Omega\times (0,T)\times D
  \end{gather*}
  and
  \[
    u(t) + \int_0^t\Delta^2u(s)\,ds - \int_0^t\Delta\left(\xi(s)+ \pi(u(s)) + g(s)\right)\,ds =
    u_0 + \int_0^tB(s,u(s))\,dW_s
  \]
  for every $t\in[0,T]$, $\P$-almost surely.
\end{defin}
\begin{rmk} 
Note that if $(u,\xi)$ is a strong solution in the sense specified above, then,
setting $w:=-\Delta u + \xi + \pi(u) +g$, we have that
\[
  u(t)-\int_0^t\Delta w(s)\,ds = u_0 + \int_0^tB(s,u(s))\,dW_s \qquad\forall\,t\in[0,T]\,,\quad \P\text{-a.s.}
\]
However, since the regularity of $\xi$, hence of $w$, is just $L^1$, the corresponding variational formulation of the 
problem is given by
\[
  \int_Du(t)\varphi - \int_0^t\!\!\int_Dw(s)\Delta\varphi\,ds = \int_Du_0\varphi + \int_D\!\!\left(\int_0^tB(s,u(s))\,dW_s\right)\varphi
  \qquad\forall\,\varphi\in \bar V_2\,.
\]
In order to get a more natural variational formulation of the problem involving the gradient of $w$, 
we shall need a regularity result on the solutions.
\end{rmk}

We are now ready to state the main results of the paper:
the first ones are an existence and continuous dependence result for problem \eqref{eq1}--\eqref{eq4}
in the case of additive and multiplicative noise.
\begin{thm}[Existence, additive noise]
  \label{thm1}
  Under the assumptions \eqref{g}--\eqref{add2}, there exists a strong solution
  to the system \eqref{eq1}--\eqref{eq4}.
\end{thm}
\begin{thm}[Continuous dependence, additive noise]
  \label{thm2}
  Let the data $(u_0^i, g_i, B_i)$ satisfy the assumptions \eqref{g}--\eqref{add2}, for $i=1,2$, such that
  \[
  (u_0^1)_D+(B_1\cdot W)_D = (u_0^2)_D + (B_2\cdot W)_D \qquad\text{a.e.~in } \Omega\times(0,T)\,,
  \]
  and let $(u_i,\xi_i)$ be any respective strong
  solution to \eqref{eq1}--\eqref{eq4}, for $i=1,2$. Then, 
  \[\begin{split}
  &\norm{u_1-u_2}_{L^2(\Omega; C^0([0,T]; V_1^*))} + \norm{\nabla(u_1-u_2)}_{L^2(\Omega; L^2(0,T; H))}\\
  &\lesssim \norm{u_0^1-u_0^2}_{L^2(\Omega; V_1^*)} + \norm{g_1-g_2}_{L^2(\Omega; L^2(0,T; V_1^*))}
  +\norm{B_1-B_2}_{L^2(\Omega; L^2(0,T; \cL_2(U,V_1^*)))}\,.
  \end{split}\]
  In particular, if $(u_0^1, g_1, B_1)=(u_0^2,g_2,B_2)$ then $u_1=u_2$ and $\xi_1-(\xi_1)_D=\xi_2-(\xi_2)_D$;
  moreover, if $\beta$ is also single-valued, then $\xi_1=\xi_2$ as well.
\end{thm}
\begin{thm}[Existence, multiplicative noise]
  \label{thm3}
  Under the assumptions \eqref{g}--\eqref{u0} and  \eqref{mult1}--\eqref{mult3}, there exists a strong solution
  to the system \eqref{eq1}--\eqref{eq4}.
\end{thm}
\begin{thm}[Continuous dependence, multiplicative noise]
  \label{thm4}
  Let the data 
  $(u_0^i,g_i)$ and $B$ satisfy the assumptions \eqref{g}--\eqref{u0} and \eqref{mult1}--\eqref{mult3}, for $i=1,2$, such that
  \[
  (u_0^1)_D = (u_0^2)_D \qquad\P\text{-a.s.}
  \]
  and let $(u_i,\xi_i)$ be any respective strong
  solution to \eqref{eq1}--\eqref{eq4}, for $i=1,2$. Then, 
  \[\begin{split}
  \norm{u_1-u_2}_{L^2(\Omega; C^0([0,T]; V_1^*))} &+ \norm{\nabla(u_1-u_2)}_{L^2(\Omega; L^2(0,T; H))}\\
  &\lesssim \norm{u_0^1-u_0^2}_{L^2(\Omega; V_1^*)} + \norm{g_1-g_2}_{L^2(\Omega; L^2(0,T; V_1^*))}\,.
  \end{split}\]
  In particular, if $(u_0^1, g_1)=(u_0^2,g_2)$ then $u_1=u_2$ and $\xi_1-(\xi_1)_D=\xi_2-(\xi_2)_D$;
  moreover, if $\beta$ is also single-valued, then $\xi_1=\xi_2$ as well.
\end{thm}
\begin{rmk}
  Note that the formulation of the continuous dependence results under the requirement that 
  the initial data have the same mean value is not new in the setting of deterministic Cahn-Hilliard systems: 
  see for example \cite{col-gil-spr, col-fuk-eqCH} for the general case of dynamic boundary conditions.
  In our setting, the natural generalization of this fact (due to the presence of the stochastic component) 
  is exactly the one that we have formulated in Theorems~\ref{thm2} and \ref{thm4}.
\end{rmk}

Finally, the last theorem that we state is a regularity result. Indeed, as we have pointed out above,
the existence Theorems~\ref{thm1} and \ref{thm3}
do not ensure any natural regularity except $L^1$ for the chemical potential: this implies that 
if we write the variational formulation of the problem, we are forced to use test functions at least in $\bar V_2$.
Since in the physical interpretation 
of the system \eqref{eq1}--\eqref{eq4} $w$ plays an important role, it is worth noting that with further
assumptions on the data $u_0$, $B$, $g$ and $j$, then $w$ is found to be more regular and a precise (and more natural) 
variational formulation of the original system can be written.
We collect these results here.
\begin{thm}[Regularity]
  \label{thm5}
  Assume that 
  \begin{gather*}
  \beta:\erre\to\erre \quad\text{locally Lipschitz continuous}\,,\\
  g \in L^2(\Omega; L^2(0,T; V_1)) \quad\text{progressively measurable}\,,\\
  u_0 \in L^2(\Omega,\cF_0, \P; V_1)\,, \qquad j(\alpha u_0)\in L^1(\Omega\times D)\quad\forall\,\alpha>0\,,\\
  B \in L^2\left(\Omega; L^2(0,T; \cL_2(U,V_{1}\cap H_0))\cap
  L^\infty(0,T; \cL_2(U,V_1^*))\right) \quad\text{progressively measurable}\,,
  \end{gather*}
 and that either one of the following conditions is satisfied:
  \begin{gather}
    \label{ip_j_pol}
    \begin{cases}
    B \in L^\infty(\Omega\times(0,T); \cL_2(U,H))\oplus L^2(0,T; L^\infty(\Omega; \cL_2(U,V_1)))\,,\\
    \exists\,R>0:\quad \beta'(x)\leq R\left(1+ |x|^2\right)\quad\text{for a.e.~}x\in\erre\,,\\
    \end{cases}\\
    \label{ip_j2}
    \begin{cases}
    N=2\,,\\
    \exists\,s>1: \quad B \in L^2(0,T; L^\infty(\Omega; \cL_2(U,V_{s})))\,,\\
    \exists\,R>0:\quad \beta'(x)\leq R\left(1+j(x)\right) \quad\text{for a.e.~}x\in\erre\,,
    \end{cases}\\
    \label{ip_j3}
    \begin{cases}
      N=3\,,\\
      \exists\,s>\frac32: \quad B \in L^2(0,T; L^\infty(\Omega; \cL_2(U,V_{s})))\,,\\
      \exists\,R>0:\quad \beta'(x)\leq R\left(1+j(x)\right) \quad\text{for a.e.~}x\in\erre\,.
    \end{cases}
  \end{gather}
  Let $(u,\xi)$ be the strong solution to \eqref{eq1}--\eqref{eq4}
  and set $w:=-\Delta u + \xi +\pi(u) + g$. Then
  \begin{gather*}
  u\in L^2(\Omega; C^0([0,T]; H))\cap L^2(\Omega; L^\infty(0,T; V_1))\cap L^2(\Omega; L^2(0,T; V_2))\,,\\
  u-B\cdot W \in L^2(\Omega; H^1(0,T; V_1^*))\,,\\
  w \in L^2(\Omega; L^2(0,T; V_1))\,, \qquad \xi\in L^2(\Omega; L^2(0,T; H))
  \end{gather*}
  and the following variational formulation of \eqref{eq1}--\eqref{eq4} holds:
  \begin{gather*}
    \sp{\partial_t(u-B\cdot W)(t)}\varphi_{V_1} + \int_D\nabla w(t)\cdot\nabla\varphi=0 \qquad\forall\,\varphi\in V_1\,,\\
    \int_Dw(t)\varphi = \int_D\nabla u(t)\cdot\nabla\varphi + \int_D\xi(t)\varphi + \int_D\pi(u(t))\varphi + \int_Dg(t)\varphi
    \qquad\forall\,\varphi\in V_1\,,
  \end{gather*}
  for almost every $t\in(0,T)$, $\P$-almost surely.
\end{thm}
\begin{rmk}\label{yosida}
  Let us comment on the growth assumptions on $\beta$ and $j$ in \eqref{ip_j_pol}--\eqref{ip_j3}:
  we show that $\beta$ is locally Lipschitz and satisfies \eqref{ip_j_pol}--\eqref{ip_j3} if and only if
  its Yosida approximations satisfy, respectively,
  \beq\label{ip_yos}
  \beta_\lambda'(x)\leq R(1+|x|^2) \quad\text{or}\quad \beta_\lambda'(x)\leq R(1+j_\lambda(x)) \qquad\text{for a.e.~}x\in\erre\,,
  \quad\forall\,\lambda\in(0,1]\,.
  \eeq
  Indeed, under \eqref{ip_j_pol}--\eqref{ip_j3}, condition \eqref{ip_yos} follows from the computation
  \[
    \beta_\lambda'(x)=\frac{\beta'((I+\lambda\beta)^{-1}x)}{1+\lambda\beta'((I+\lambda\beta)^{-1}x)}\leq
    \beta'((I+\lambda\beta)^{-1}x)\lesssim 1+j((I+\lambda\beta)^{-1}x)
  \]
  together with the contraction property of the resolvent $(I+\lambda\beta)^{-1}$ and the fact that 
  $j((I+\lambda\beta)^{-1})\leq j_\lambda$.
  Viceversa, if $\beta:\erre\to2^\erre$ is a maximal monotone graph with $D(\beta)=\erre$
  satisfying \eqref{ip_yos},
  then $\beta$
  is single-valued,
  $\beta\in W^{1,\infty}_{loc}(\erre)$ (or in other words $\beta:\erre\to\erre$ is locally Lipschitz continuous)
  and satisfies \eqref{ip_j_pol}--\eqref{ip_j3}.
  Indeed, for any compact interval $I\subseteq\erre$, 
  $(\beta_\lambda')$ and $(\beta_\lambda)$ are uniformly bounded in $L^\infty(I)$ thanks to
  any of \eqref{ip_yos} and the 
  hypothesis $D(\beta)=\erre$, respectively.
  Since $\beta_\lambda\nearrow\beta^0$, as well-known, we deduce that $\beta^0\in W^{1,\infty}(I)$;
  moreover, recalling that $\beta^0$ is discontinuous at any point in which $\beta$ is multivalued,
  we also infer that $\beta$ is singlevalued and $\beta=\beta_0\in W^{1,\infty}(I)$, hence $\beta\in W^{1,\infty}_{loc}(\erre)$.
  Furthermore, \eqref{ip_yos} implies that
  \[
    \int_I\beta_\lambda'(x)\varphi(x)\,dx\leq \int_IR(1+|x|^2)\varphi(x)\,dx \quad\text{or}\quad
     \int_I\beta_\lambda'(x)\varphi(x)\,dx\leq \int_IR(1+j_\lambda(x))\varphi(x)\,dx
  \]
  for every $\varphi\in L^1(I)$, $\varphi\geq0$. Hence, since $\beta_\lambda'\wstarto\beta'$ in $L^\infty(I)$, letting
  $\lambda\searrow0$ yields
  \[
    \int_I\beta'(x)\varphi(x)\,dx\leq \int_IR(1+|x|^2)\varphi(x)\,dx \quad\text{or}\quad
     \int_I\beta'(x)\varphi(x)\,dx\leq \int_IR(1+j(x))\varphi(x)\,dx
  \]
  for every $\varphi\in L^1(I)$, $\varphi\geq0$, from which the conditions in \eqref{ip_j_pol}--\eqref{ip_j3}.
\end{rmk}
\begin{rmk}
  Let us comment on the assumptions \eqref{ip_j_pol}--\eqref{ip_j3}. First of all, 
  it is worth pointing out that a large class of operators $B$
  satisfying the hypotheses above is composed by all the nonrandom and
  time-independent operators $B$ in $\cL_2(U, V_s\cap H_0)$ ($s=1$, $s>1$ or $s>3/2$, respectively).
  Secondly, note that \eqref{ip_j_pol} we are assuming that $\beta'$
  are uniformly bounded by a polynomial of degree $2$: this condition is very natural in the setting of Cahn-Hilliard systems 
  and is met when $\beta$ is a polynomial of degree $3$, which corresponds to the classical case of a double-well
  potential of degree $4$. On the other side, let us stress that 
  in \eqref{ip_j2}--\eqref{ip_j3} the requirement for $\beta$ is much 
  weaker and is met by a very large class of functions: for example,
  any function $\beta$ that can be written as a polynomial (of arbitrary degree) times
  a first order exponential. The possibility of taking into account also singular potentials in the context of regularity
  for stochastic Cahn-Hilliard systems is very interesting and provides a significant improvement 
  to the existing results in literature, which deal essentially with the classical case of a polynomial double-well
  potential.
\end{rmk}


\section{Well-posedness with additive noise}
\setcounter{equation}{0}
\label{additive}

In this section, we present the proof of Theorems~\ref{thm1}--\ref{thm2}.
The proof is divided in several steps: we approximate the nonlinearity of the problem using classical
Yosida regularizations, we prove uniform estimates on the approximated solutions and then we pass to the limit
in suitable topologies through monotonicity and compactness arguments.

We shall firstly assume that 
\beq\label{strong_B}
  B \in L^2(\Omega; L^2(0,T; \cL_2(U, \bar V)))\,,
\eeq
where $\bar V$ is a separable Hilbert space continuously embedded in $\bar V_2$:
for example, by the Sobolev embeddings theorems
we can take $\bar V:= \{v\in H^m(D): \partial_{\bf n}v=0 \text{ in } \Gamma\}$, 
where $m>2+\frac{N}{2}$. This hypothesis will be removed at the end of the section.

\subsection{The approximated problem}
\label{approximation}
For any $\lambda\in(0,1]$, let $\beta_\lambda:\erre\to\erre$ be the Yosida approximation of $\beta$ and
$j_\lambda:\erre\to[0,+\infty)$ the Moreau-Yosida regularization of $j$. We consider the approximated problem
\begin{align}
  \label{eq1_app}
  du_\lambda(t) - \Delta w_\lambda(t)\,dt = B(t)\,dW_t \qquad&\text{in } (0,T)\times D\,,\\
  \label{eq2_app}
  w_\lambda = -\Delta u_\lambda + \beta_\lambda(u_\lambda) + \pi(u_\lambda) + g \qquad&\text{in } (0,T)\times D\,,\\
  \label{eq3_app}
  \partial_{\bf n} u_\lambda=0\,, \quad \partial_{\bf n}w_\lambda=0 \qquad&\text{in } (0,T)\times\Gamma\,,\\
  \label{eq4_app}
  u_\lambda(0)=u_0 \qquad&\text{in } D\,,
\end{align}
In order to solve \eqref{eq1_app}--\eqref{eq4_app}, as we have done in the previous section,
we first need to identify a possible variational formulation
of the approximated problem. To this aim, if we introduce the operator $\mathcal A_\lambda$ as 
\begin{gather*}
  D(\mathcal A_\lambda):=\left\{(\omega,t,v)\in\Omega\times[0,T]\times V_2: \; 
  -\Delta v + \beta_\lambda(v) + \pi(v)+g(\omega,t) \in V_2\right\}\,,\\
  \mathcal A_\lambda (\omega,t,v):= -\Delta(-\Delta v + \beta_\lambda(v) + \pi(v) + g(\omega,t))\,, \qquad v\in D(\mathcal A_\lambda)\,,
\end{gather*}
then the problem \eqref{eq1_app}--\eqref{eq4_app} can be written (formally) in terms only of the variable $u_\lambda$ as
\[
  d u_\lambda(t) + \mathcal A_\lambda (t,u_\lambda(t))\,dt = B(t)\,dW_t\,, \qquad u_\lambda(0)=u_0\,.
\]
Now, in order to get a weak formulation of our problem, 
note that for any $\varphi \in V_2$ and $(\omega,t,v)\in D(\mathcal A_\lambda)$, by integration by parts we have that
\[\begin{split}
\left(\mathcal A_\lambda (\omega,t,v), \varphi\right)_H&=\int_D(-\Delta v + \beta_\lambda(v) + \pi(v) + g(\omega,t))(-\Delta\varphi)\\
&=\int_D\Delta v\Delta\varphi - \int_D\beta_\lambda(v)\Delta\varphi - \int_D\pi(v)\Delta\varphi - \int_D g(\omega,t)\Delta\varphi\\
&\leq\left(\norm{\Delta v}_H + \frac1\lambda\norm{v}_H + C_\pi\norm{v}_H+ \norm{g(\omega,t)}_H\right)\norm{\Delta\varphi}_H\,.
\end{split}\]
This implies that $\mathcal A_\lambda$ can be extended to a variational operator 
$A_\lambda: \Omega\times[0,T]\times V_2\to V_2^*$,
\[
  \sp{A_\lambda(\omega,t,v)}{\varphi}_{V_2}:=
  \int_D\Delta v\Delta\varphi - \int_D\beta_\lambda(v)\Delta\varphi - \int_D\pi(v)\Delta\varphi - \int_D g(\omega,t)\Delta\varphi\,,
\]
for $v,\varphi\in V_2$ and $(\omega,t)\in \Omega\times[0,T]$.
Note that with the notation that we have introduced
we have the representation $A_\lambda=\Delta^2-\Delta\beta_\lambda -\Delta\pi-\Delta g$.
Hence, the idea is to study the approximated problem \eqref{eq1_app}--\eqref{eq4_app} using the classical
variational theory by Pardoux and Krylov-Rozovski{\u\i} in the Hilbert triplet $(V_2,H,V_2^*)$. We need the following lemma.

\begin{lem}
  \label{prop_A}
  Let $\lambda\in(0,1]$. Then, the operator $A_\lambda:\Omega\times[0,T]\times V_2\to V_2^*$
  is progressively measurable, hemicontinuous, weakly monotone, weakly coercive and bounded, i.e.
  \begin{itemize}
    \item the map
    \[
    \eta\mapsto \sp{A_\lambda(\omega,t,v_1+\eta v_2)}{\varphi}_{V_2}\,, \qquad \eta\in\erre\,,
    \]
    is continuous for every 
    $(\omega,t)\in\Omega\times[0,T]$ and $v_1,v_2,\varphi\in V_2$;
    \item there exists a constant $c>0$ such that 
    \[
    \sp{A_\lambda (\omega,t,v_1)-A_\lambda (\omega,t,v_2)}{v_1-v_2}_{V_2}\geq -c\norm{v_1-v_2}_H^2
    \]
    for every 
    $(\omega,t)\in\Omega\times[0,T]$ and $v_1,v_2\in V_2$;
    \item there exist two constants $c_1,c_1'>0$ and an adapted process $f_1 \in L^1(\Omega\times(0,T))$ 
    such that 
    \[
    \sp{A_\lambda (\omega,t,v)}{v}_{V_2}\geq c_1\norm{v}^2_{V_2} - c_1'\norm{v}^2_H - f_1(\omega,t)
    \]
    for every
    $(\omega,t)\in\Omega\times[0,T]$ and $v\in V_2$; 
    \item there exists a constant $c_2>0$ and an adapted process $f_2 \in L^2(\Omega\times(0,T))$ such that
    \[
    \norm{A_\lambda (\omega,t,v)}_{V_2^*}\leq c_2\norm{v}_{V_2} + f_2(\omega,t)
    \]
    for every $(\omega,t)\in\Omega\times[0,T]$ and $v\in V_2$.
  \end{itemize}
\end{lem}
\begin{proof}
  First of all, 
  $A_\lambda$ is progressively measurable since so is $g$. Moreover,
  for every $(\omega,t)\in\Omega\times[0,T]$ and $v_1,v_2,\varphi\in V_2$, we have
  \[\begin{split}
  \sp{A_\lambda(\omega,t,v_1+\eta v_2)}\varphi_{V_2}&=\int_D\Delta v_1\Delta\varphi+\eta\int_D\Delta v_2\Delta\varphi\\
  &-\int_D\beta_\lambda(v_1+\eta v_2)\Delta\varphi - \int_D\pi(v_1+\eta v_2)\Delta\varphi - \int_Dg(\omega,t)\Delta\varphi\,,
  \end{split}\]
  which is continuous in $\eta\in\erre$ by the Lipschitz-continuity of $\beta_\lambda$ and $\pi$.
  Secondly, 
  for every $(\omega,t)\in \Omega\times[0,T]$ and $v_1,v_2\in V_2$,
  an easy computation based on the H\"older and Young inequalities shows that
  \[\begin{split}
  &\sp{A_\lambda (\omega,t,v_1)-A_\lambda (\omega,t,v_2)}{v_1-v_2}_{V_2}\\
  &=\int_D|\Delta(v_1-v_2)|^2-\int_D(\beta_\lambda(v_1)-\beta_\lambda(v_2))\Delta(v_1-v_2)
  -\int_D(\pi(v_1)-\pi(v_2))\Delta(v_1-v_2)\\
  &\geq\norm{\Delta(v_1-v_2)}_{H}^2-\left(\frac1\lambda+C_\pi\right)\norm{v_1-v_2}_H\norm{\Delta(v_1-v_2)}_{H}\\
  &\geq
  \frac12\norm{v_1-v_2}_{V_2}^2 - \frac12\left(1+\left(\frac1\lambda+C_\pi\right)^2\right)\norm{v_1-v_2}_H^2
  \end{split}\]
  from which the weak monotonicity with the choice $c=\frac12+\frac12(\frac1\lambda + C_\pi)^2$.
  Similarly, for every $(\omega,t)\in \Omega\times[0,T]$ and $v\in V_2$,
  by the H\"older inequality 
  and the Lipschitz continuity of $\beta_\lambda$ and $\pi$,
  \[\begin{split}
  &\sp{A_\lambda (\omega,t,v)}{v}_{V_2}=
  \int_D|\Delta v|^2-\int_D\beta_\lambda(v)\Delta v
  -\int_D\pi(v)\Delta v - \int_D g(\omega,t)\Delta v\\
  &\qquad\qquad
  \geq\norm{\Delta v}_{H}^2-\left(\frac1\lambda+C_\pi\right)\norm{v}_H\norm{\Delta v}_{H}-C_0|D|^{1/2}\norm{\Delta v}_{H}
  -\norm{g(\omega,t)}_H\norm{\Delta v}_{H}\\
  &\qquad\qquad\geq\frac14\norm{v}_{V_2}^2 - \left(\frac14+\left(\frac1\lambda+C_\pi\right)^2\right)\norm{v}_H^2 
  -C_0^2|D|-\norm{g(\omega,t)}_H^2
  \end{split}\]
  so that $c_1=\frac14$, $c_1'=\frac14+(\frac1\lambda+C_\pi)^2$ and 
  $f_1=C_0^2|D| + \norm{g}_H^2\in L^1(\Omega\times(0,T))$ is a possible choice.
  Finally, for every $(\omega,t)\in \Omega\times[0,T]$ and $v,\varphi\in V_2$, by the H\"older inequality 
  and the Lipschitz continuity of $\beta_\lambda$ and $\pi$ we have
  \[\begin{split}
    \sp{A_\lambda (\omega,t,v)}{\varphi}_{V_2}&=
    \int_D\Delta v\Delta\varphi - \int_D\beta_\lambda(v)\Delta\varphi - \int_D\pi(v)\Delta\varphi - \int_D g(\omega,t)\Delta\varphi\\
    &\leq\left(\norm{\Delta v}_H + \left(\frac1\lambda+C_\pi\right)\norm v_H + C_0|D|^{1/2}+\norm{g(\omega,t)}_H\right)
    \norm{\varphi}_{V_2}
  \end{split}\]
  which implies that 
  \[
  \norm{A_\lambda(\omega,t,v)}_{V_2^*}\leq \left(1+\frac1\lambda + C_\pi\right)\norm{v}_{V_2} + 
  C_0|D|^{1/2}+\norm{g(\omega,t)}_H\,,
  \]
  which proves that $A_\lambda$ is also bounded
  with the choice of coefficients $c_2:=1+\frac1\lambda + C_\pi$
  and $f_2=C_0|D|^{1/2} + \norm{g}_H\in L^2(\Omega\times(0,T))$.
\end{proof}

By Lemma~\ref{prop_A}, the classical variational theory by Pardoux and Krylov-Rozovski{\u\i}
(see \cite{Pard, KR-spde})
ensures that there exists a unique process 
\[
  u_\lambda \in L^2(\Omega; C^0([0,T]; H)) \cap L^2(\Omega; L^2(0,T; V_2))
\]
such that
\beq
\label{app_eq}
  u_\lambda(t) + \int_0^t A_\lambda(s, u_\lambda(s))\,ds = u_0 + \int_0^t B(s)\,dW_s \qquad\forall\,t\in[0,T]\,, \quad\P\text{-a.s.}
\eeq

\subsection{The first estimate}
\label{first_est}
It\^o's formula yields
\[\begin{split}
  \frac12\norm{u_\lambda(t)}_H^2 &+ \int_0^t\norm{\Delta u_\lambda(s)}_H^2\,ds + 
  \int_0^t\!\!\int_D\beta_\lambda'(u_\lambda(s))|\nabla u_\lambda(s)|^2\,ds\\
  &=\frac12\norm{u_0}_H^2 + \int_0^t\!\!\int_D\pi(u_\lambda(s))\Delta u_\lambda(s)\,ds
  +\int_0^t\!\!\int_Dg(s)\Delta u_\lambda(s)\,ds\\
  &+\frac12\int_0^t\norm{B(s)}^2_{\cL_2(U,H)}\,ds + \int_0^t\left(u_\lambda(s), B(s)\right)_H\,dW_s
\end{split}\]
and by the Young inequality, the monotonicity of $\beta_\lambda$ and the Lipschitzianity of $\pi$ we deduce that
\[\begin{split}
  \frac12\norm{u_\lambda(t)}_H^2 &+ \frac12\int_0^t\norm{\Delta u_\lambda(s)}_H^2\,ds \leq
  \frac12\norm{u_0}_H^2 + \norm{g}^2_{L^2(0,T;H)} + 2T|D|C_0^2 \\
  &+2C_\pi^2\int_0^t\norm{u_\lambda(s)}_H^2\,ds +\frac12\int_0^T\norm{B(s)}^2_{\cL_2(U,H)}\,ds +
   \int_0^t\left(u_\lambda(s), B(s)\right)_H\,dW_s\,.
\end{split}\]
Thanks to the Gronwall lemma, 
taking supremum in time and expectations, using the Burkholder-Davis-Gundy inequality on the last terms on the
right-hand side and the Young inequality, we get, for any $\eps>0$,
\[\begin{split}
  \E\sup_{t\in[0,T]}\norm{u_\lambda(t)}_H^2 &+ \E\int_0^T\norm{u_\lambda(s)}_2^2\,ds \lesssim_{T,|D|, C_0} 1+
  \norm{u_0}_{L^2(\Omega;H)}^2 + \norm{g}^2_{L^2(\Omega; L^2(0,T;H))} \\
  & +\norm{B}^2_{L^2(\Omega; L^2(0,T;\cL_2(U,H)))} +
   \E\sup_{t\in[0,T]}\abs{\int_0^t\left(u_\lambda(s), B(s)\right)_H\,dW_s}\\
   &\lesssim 1+
  \norm{u_0}_{L^2(\Omega;H)}^2 + \norm{g}^2_{L^2(\Omega; L^2(0,T;H))} +\norm{B}^2_{L^2(\Omega; L^2(0,T;\cL_2(U,H)))}\\
  &\qquad+\E\left(\int_0^T\norm{u_\lambda(s)}_H^2\norm{B(s)}_{\cL_2(U,H)}^2\,ds\right)^{1/2}\\
  &\leq1+
  \norm{u_0}_{L^2(\Omega;H)}^2 + \norm{g}^2_{L^2(\Omega; L^2(0,T;H))} +\norm{B}^2_{L^2(\Omega; L^2(0,T;\cL_2(U,H)))}\\
  &\qquad+\eps\E\sup_{t\in[0,T]}\norm{u_\lambda(t)}_H^2 + \frac1{4\eps}\norm{B}^2_{L^2(\Omega; L^2(0,T;\cL_2(U,H)))}\,.
\end{split}\]
Choosing $\eps$ small enough, we deduce that there exists a constant $M$, independent of $\lambda$, such that
\beq
\label{est1}
  \norm{u_\lambda}_{L^2(\Omega; L^\infty(0,T; H))}^2 + \norm{u_\lambda}^2_{L^2(\Omega; L^2(0,T; V_2))}
  \leq M
  \qquad\forall\,\lambda\in(0,1]\,.
\eeq

\subsection{The second estimate}
\label{second_est}
Let us prove a further estimate on the approximated solutions. To this end, we need some preparatory work.
We introduce the convex function 
\[
  \Phi:H\to[0,+\infty)\,, \qquad \Phi(x):=\frac12\norm{\nabla{\mathcal N}(x-x_D)}_H^2\,, \quad x\in H\,.
\]
Let us show that the subdifferential of $\Phi$ is given by $\partial\Phi(x)=\mathcal N(x-x_D)$, for $x\in H$: 
indeed, by definition of $\mathcal N$, for every $x,y\in H$ we have
\[\begin{split}
  \frac12\norm{\nabla\mathcal{N}(x-x_D)}_H^2 &+ \left(\mathcal N(x-x_D), y-x\right)_H\\
  &=\frac12\norm{\nabla\mathcal N(x-x_D)}_H^2 + \left(\mathcal N(x-x_D), (y-y_D)-(x-x_D)\right)_H \\
  &=\frac12\norm{\nabla\mathcal N(x-x_D)}_H^2 + \int_D\nabla\mathcal N(x-x_D)\cdot\nabla\mathcal N\left((y-y_D)-(x-x_D)\right)\\
  &\leq\frac12\norm{\nabla\mathcal N(y-y_D)}_H^2\,,
\end{split}\]
so that $\mathcal N(x-x_D)\in\partial\Phi(x)$. Since $x\mapsto \mathcal N(x-x_D)$, $x\in H$, is maximal monotone on $H$, the claim
is proved. In particular, we have that $\Phi\in C^2(H)$.

Now, in order to get some weak compactness for $\beta_\lambda(u_\lambda)$, the idea is first
to recover an $L^1$ estimate for the product $\beta_\lambda(u_\lambda)u_\lambda$.
To this end, we should need an It\^o's formula for $\Phi(u_\lambda)$: however, this is not obvious,
since in the variational framework,
only It\^o's formula for the square of the $H$-norm is known. The idea is to smooth out \eqref{app_eq} through
the resolvent of the linear operator $\Delta^2$ in order to recover an equation in $H$, apply the classical It\^o's
formula for the function $\Phi$ on $H$ and then pass to the limit in the approximation.
A possible way to avoid this further approximation would be to look at 
\eqref{app_eq} as an equation on $V_1^*$, note that $\Phi$ can be extended
to $V_1^*$ and use It\^o's formula directly on $V_1^*$: however, we prefer here the first approach since it is more
constructive.

We shall use the notation $B_D$ for the operator
\[
  B_D:\Omega\times[0,T]\to \cL_2(U,\erre)\,, \qquad B_D(\omega,t)x:=\left(B(\omega,t)x\right)_D\,,
  \quad (\omega,t, x)\in \Omega\times[0,T]\times U
\]
and note that $\norm{B_D}_{\cL_2(U,\erre)}\leq \sqrt{|D|}\norm{B}_{\cL_2(U,H)}$
in $\Omega\times[0,T]$. Setting also
\[
  m(t):=(u_0)_D + \left(B\cdot W(t)\right)_D=(u_0)_D+\int_0^tB_D(s)\,dW_s\,, \qquad t\in[0,T]\,,
\]
since $B_D \in L^2(\Omega; L^2(0,T; \cL_2(U,\erre)))$, then we have $m \in L^2(\Omega; C^0([0,T]))$.
With this notation, it is a standard matter to check that \eqref{u0} and \eqref{add2} imply that 
\beq
  \label{int_m}
  j(\alpha m) \in L^1(\Omega\times(0,T)) \qquad\forall\,\alpha>0\,.
\eeq

First of all, note that testing \eqref{app_eq} by the constant $\frac1{|D|}$, we have that 
$(u_\lambda(t))_D=m(t)$ for every $t\in[0,T]$ and $\lambda \in (0,1]$.
Secondly, we consider for any $\delta>0$ the operator
\[
  (I+\delta\Delta^2)^{-1}:V_2^*\to V_2
\]
and note that, as $\delta\searrow0$, $(I+\delta\Delta^2)^{-1}x\to x$ in $V_2^*$ (respectively in $V_2$) for every $x\in V_2^*$ 
(respectively $x\in V_2$). Moreover, using the definition of resolvent, it is a standard matter to check that 
$(I+\delta\Delta^2)^{-1}$ and $-\Delta$ commute on $H$.
Hence, using the notation $h^\delta:=(I+\delta \Delta^2)^{-1}h$ for every $h\in V_2^*$, applying $(I+\delta\Delta^2)^{-1}$ to 
\eqref{app_eq} we get
\[\begin{split}
  u_\lambda^\delta(t) &+ \int_0^t\Delta^2u_\lambda^\delta(s)\,ds - \int_0^t\Delta(\beta_\lambda(u_\lambda(s)))^\delta\,ds
  -\int_0^t\Delta(\pi(u_\lambda(s)))^\delta\,ds-\int_0^t\Delta g^\delta(s)\,ds\\
  &=u_0^\delta + \int_0^tB^\delta(s)\,dW_s \qquad\text{in } H\,, \quad\forall\,t\in[0,T]\,, \quad\P\text{-a.s.}
\end{split}\]
Now, since $(I+\delta \Delta^2)^{-1}$ preserves the mean, we have $(u_\lambda^\delta)_D=m$.
Taking into account that the previous equation holds in $H$, the classical It\^o's formula for $\Phi(u_\lambda^\delta)$
(see \cite[Thm.~4.32]{dapratozab}) yields
\[\begin{split}
  &\frac12\norm{\nabla\mathcal N(u^\delta_\lambda(t)-m(t))}^2_H + 
  \int_0^t\left(\Delta^2u_\lambda^\delta(s), \mathcal N(u^\delta_\lambda(s)-m(s))\right)_H\,ds\\
  &-\int_0^t\left(\Delta(\beta_\lambda(u_\lambda(s)))^\delta, \mathcal N(u^\delta_\lambda(s)-m(s))\right)_H\,ds\\
  &-\int_0^t\left(\Delta(\pi(u_\lambda(s)))^\delta, \mathcal N(u^\delta_\lambda(s)-m(s))\right)_H\,ds
  -\int_0^t\left(\Delta g^\delta(s), \mathcal N(u^\delta_\lambda(s)-m(s))\right)_H\,ds\\
  &=\frac12\norm{\nabla \mathcal N(u^\delta_0-(u^\delta_0)_D)}_H^2
  +\frac12\int_0^t\operatorname{Tr}\left[(B^\delta)^*(s)\mathcal N(\cdot-(\cdot)_D)B^\delta(s)\right]\,ds\\
  &+\int_0^t\left(\mathcal N(u^\delta_\lambda(s)-m(s)), B^\delta(s)\right)_H\,dW_s \qquad\forall\,t\in[0,T]\,, \quad\P\text{-a.s.}
\end{split}\]
Now, by the properties of $\norm\cdot_*$ and the continuous inclusion $H\embed V_1^*$, it is not difficult to check that
\[
  \norm{\nabla \mathcal N(u^\delta_0-(u^\delta_0)_D)}_H^2=
  \norm{u^\delta_0-(u^\delta_0)_D)}_{*}^2\lesssim \norm{u_0^\delta}_H^2\leq \norm{u_0}_H^2\,.
\]
Moreover, if $\{e_k\}_{k\in\enne}$ is an orthonormal system of $U$, 
again by definition of $\mathcal N$, condition \eqref{N1} and the continuous inclusion $H\embed V_1^*$ we have
\[\begin{split}
\operatorname{Tr}&\left[(B^\delta)^*\mathcal N(\cdot-(\cdot)_D)B^\delta\right]=
\sum_{k=0}^\infty\left(\mathcal N(B^\delta e_k-(B^\delta e_k)_D), B^\delta e_k\right)_H\\
&=\sum_{k=0}^\infty\left(\mathcal N(B^\delta e_k-(B^\delta e_k)_D), (B^\delta e_k-(B^\delta e_k)_D) \right)_H=
\sum_{k=0}^\infty\norm{\nabla\mathcal N(B^\delta e_k-(B^\delta e_k)_D)}^2_H\\
&=\sum_{k=0}^\infty\norm{(B^\delta e_k-(B^\delta e_k)_D)}^2_{*}\lesssim
\sum_{k=0}^\infty\norm{B^\delta e_k}_H^2\leq\sum_{k=0}^\infty\norm{B e_k}_H^2=\norm{B}^2_{\cL_2(U,H)}\,.
\end{split}\]
Finally, by the Burkholder-Davis-Gundy,  the Young inequalities, the continuous inclusion $H\embed V_1^*$ and
the fact that $\mathcal N(u_\delta-m)$ has null mean, we have
\[\begin{split}
\E&\sup_{t\in[0,T]}\abs{\int_0^t\left(\mathcal N(u^\delta_\lambda(s)-m(s)), B^\delta(s)\right)_H\,dW_s}\\
&\lesssim
\E\left(\int_0^T\norm{\mathcal N(u^\delta_\lambda(s)-m(s))}_{V_1}^2\norm{B^\delta(s)}_{\cL_2(U,V_1^*)}^2\,ds\right)^{1/2}\\
&\lesssim\eps\E\sup_{t\in[0,T]}\norm{\nabla\mathcal N(u^\delta_\lambda(t)-m(t))}_H^2 + 
\frac1{4\eps}\norm{B}^2_{L^2(\Omega; L^2(0,T; \cL_2(U,H)))}
\end{split}\]
for every $\eps>0$. Choosing $\eps$ small enough and taking these estimates into account we deduce that
\[\begin{split}
  \E&\sup_{t\in[0,T]}\norm{\mathcal N(u^\delta_\lambda(t)-m(t))}_H^2+
  \E\int_0^T\left(\Delta^2u_\lambda^\delta(s), \mathcal N(u^\delta_\lambda(s)-m(s))\right)_H\,ds\\
  &-\E\int_0^T\left(\Delta(\beta_\lambda(u_\lambda(s)))^\delta, \mathcal N(u^\delta_\lambda(s)-m(s))\right)_H\,ds\\
  &-\E\int_0^T\left(\Delta(\pi(u_\lambda(s)))^\delta, \mathcal N(u^\delta_\lambda(s)-m(s))\right)_H\,ds\\
  &-\E\int_0^T\left(\Delta g^\delta(s), \mathcal N(u^\delta_\lambda(s)-m(s))\right)_H\,ds
  \lesssim \norm{u_0}^2_{L^2(\Omega; H)} + \norm{B}^2_{L^2(\Omega; L^2(0,T; \cL_2(U,H)))}\,,
\end{split}\]
where the implicit constant is independent of $\delta$ and $\lambda$. Now, by the properties of the resolvent
$(I+\delta\Delta^2)^{-1}$, the continuity of $-\Delta:H\to V_2^*$ 
and the regularities of $u_\lambda$, $\beta_\lambda(u_\lambda)$, 
$\pi(u_\lambda)$ and
$g$ we have, as $\delta\searrow0$,
\begin{align*}
  u_\lambda^\delta\to u_\lambda \qquad&\text{in } L^2(\Omega; L^2(0,T; V_2))\,,\\
  \Delta^2u_\lambda^\delta\to \Delta^2u_\lambda \qquad&\text{in } L^2(\Omega; L^2(0,T; V_2^*))\,,\\
  -\Delta(\beta_\lambda(u_\lambda))^\delta \to 
  -\Delta\beta_\lambda(u_\lambda) \qquad&\text{in } L^2(\Omega; L^2(0,T; V_2^*))\,,\\
  -\Delta(\pi(u_\lambda))^\delta\to -\Delta\pi(u_\lambda) \qquad&\text{in } L^2(\Omega; L^2(0,T; V_2^*))\,,\\
  -\Delta g^\delta \to -\Delta g \qquad&\text{in } L^2(\Omega; L^2(0,T; V_2^*))\,.
\end{align*}
In particular we have that 
\[
  u_\lambda^\delta-m\to u_\lambda-m \qquad\text{in } L^2(\Omega; L^2(0,T; H_0))
\]
and consequently, by the continuity of $\mathcal N:H_0\to V_2$,
\[
  \mathcal N(u_\lambda^\delta-m) \to \mathcal N(u_\lambda-m)\qquad\text{in } L^2(\Omega; L^2(0,T; V_2))\,.
\]
Taking these remarks into account, 
letting $\delta\searrow0$ and ignoring the first term on the left-hand side, we deduce that 
\[\begin{split}
  &\E\int_0^T\sp{\Delta^2u_\lambda(s)}{\mathcal N(u_\lambda(s)-m(s))}_{V_2}\,ds
  -\E\int_0^T\sp{\Delta\beta_\lambda(u_\lambda(s))}{\mathcal N(u_\lambda(s)-m(s))}_{V_2}\,ds\\
  &-\E\int_0^T\sp{\Delta\pi(u_\lambda(s))}{\mathcal N(u_\lambda(s)-m(s))}_{V_2}\,ds
  -\E\int_0^T\sp{\Delta g(s)}{\mathcal N(u_\lambda(s)-m(s))}_{V_2}\,ds\\
  &\lesssim \norm{u_0}^2_{L^2(\Omega; H)} + \norm{B}^2_{L^2(\Omega; L^2(0,T; \cL_2(U,H)))}\,,
\end{split}\] 
which by the definition of $\mathcal N$ leads to
\[\begin{split}
  \E&\int_0^T\left(-\Delta u_\lambda(s)+\beta_\lambda(u_\lambda(s))+\pi(u_\lambda(s))+g(s), u_\lambda(s)-m(s)\right)_H\,ds\\
  &\lesssim \norm{u_0}^2_{L^2(\Omega; H)} + \norm{B}^2_{L^2(\Omega; L^2(0,T; \cL_2(U,H)))}\,.
\end{split}\]
By the Burkholder-Davis-Gundy inequality we have
\[\begin{split}
  \E\sup_{t\in[0,T]}|m(t)|^2 &\lesssim \norm{u_0}_{L^2(\Omega; H)}^2+\E\int_0^T\norm{B_D(s)}^2_{\cL_2(U,\erre)}\,ds\\
  &\lesssim_{|D|}\norm{u_0}_{L^2(\Omega; H)}^2 + \norm{B}^2_{L^2(\Omega; L^2(0,T; \cL_2(U,H)))}\,,
\end{split}\]
hence, integrating by parts and using the Lipschitz continuity of $\pi$, the Young inequality and the estimate \eqref{est1},
we deduce that
\beq\label{aux_ineq}\begin{split}
  \E\int_0^T&\norm{\nabla u_\lambda(s)}^2_H\,ds + \E\int_0^T\!\!\int_D\beta_\lambda(u_\lambda(s))u_\lambda(s)\,ds
  \lesssim_{C_\pi, C_0, |D|} 1+\norm{u_0}^2_{L^2(\Omega; H)} \\
  &+ \norm{B}^2_{L^2(\Omega; L^2(0,T; \cL_2(U,H)))} + 
  \norm{g}^2_{L^2(\Omega; L^2(0,T; H))}
  +\E\int_0^T\!\!\int_D\beta_\lambda(u_\lambda(s))m(s)\,ds\,.
\end{split}\eeq
Now, for every $\alpha>1$,
we estimate the last term on the right-hand side using the generalized Young inequality and \eqref{int_m} as
\[
  \beta_\lambda(u_\lambda)m=\frac1\alpha\beta_\lambda(u_\lambda)\alpha m\leq
  \frac1\alpha j^*(\beta_\lambda(u_\lambda)) + j(\alpha m)\,,
\]
and similarly, recalling that $\beta_\lambda \in \beta((I+\lambda\beta)^{-1})$, on the left-hand side
\[
  \beta_\lambda(u_\lambda)u_\lambda = j((I+\lambda\beta)^{-1}u_\lambda) + j^*(\beta_\lambda(u_\lambda))
  \geq j^*(\beta_\lambda(u_\lambda))\,.
\]
Rearranging the terms, we deduce that 
\[
  \E\int_0^T\!\!\int_Dj^*(\beta_\lambda(u_\lambda))\lesssim
  1+\norm{j(\alpha m)}_{L^1(\Omega\times(0,T))} + \frac{1}{\alpha}\E\int_0^T\!\!\int_D j^*(\beta_\lambda(u_\lambda))\,,
\]
where the implicit constant depends only on $u_0,B,g,\pi,D$.
Since $\alpha>1$ is arbitrary, choosing $\alpha$ sufficiently large we deduce
by condition \eqref{int_m} that $j^*(\beta_\lambda(u_\lambda))$
is bounded in $L^1(\Omega\times(0,T)\times D)$ uniformly in $\lambda$. Hence, going back in the inequality \eqref{aux_ineq} we 
infer that there exists a constant $M>0$ such that
\beq
  \label{est2}
  \norm{\beta_\lambda(u_\lambda)u_\lambda}_{L^1(\Omega\times(0,T)\times D)} \leq M \qquad\forall\,\lambda\in(0,1]\,.
\eeq

\subsection{The third estimate}
We prove now pathwise estimates on the approximated solutions (i.e.~with $\omega\in\Omega$ fixed).
To this end, note that the integral relation \eqref{app_eq} can be written as
\beq
  \label{app_eq'}
  \partial_t(u_\lambda - B\cdot W) + \Delta^2 u_\lambda -
  \Delta\left(\beta_\lambda(u_\lambda) + \pi(u_\lambda) + g\right) = 0 \quad\text{a.e.~in } (0,T)\,, \quad\P\text{-a.s.}
\eeq
Thanks to the further regularity that we have assumed on $B$ in \eqref{strong_B}, we have that 
\[
  B\cdot W \in L^2(\Omega; C^0([0,T]; \bar V))\,,
\]
so that, taking also into account \eqref{u0} and the definition of $\bar{V}$,
we can choose a set $\Omega_0\in\cF$ with $\P(\Omega_0)=1$ such that \eqref{app_eq'} holds in $\Omega_0$
and we have, for every $\omega\in\Omega_0$, that 
\begin{gather*}
  (B\cdot W)(\omega) \in L^2(0,T; V_2)\cap L^\infty(0,T; W^{2,\infty}(D))\,,\\
  u_0(\omega)\in H\,, \qquad j(\alpha (u_0(\omega))_D)<+\infty \quad\forall\,\alpha>0\,.
\end{gather*}

Fix now $\omega\in \Omega_0$: we will not write the dependence on $\omega$ as no confusion can arise.
Since $(u_\lambda-B\cdot W)_D=(u_0)_D$, 
setting $m_0:=(u_0)_D + B\cdot W$
we can test \eqref{app_eq'} by $\mathcal N(u_\lambda-m_0)$:
taking into account the definition of $\mathcal N$ and \eqref{N3}, we have
\[\begin{split}
  \frac12&\norm{\nabla\mathcal N(u_\lambda-m_0)(t)}_H^2
  -\int_0^t\!\!\int_D\Delta u_\lambda(s)(u_\lambda-m_0)(s)\,ds\\
  &+\int_0^t\!\!\int_D\left[\beta_\lambda(u_\lambda(s))
  +\pi(u_\lambda(s))+g(s)\right](u_\lambda-m_0)(s)\,ds = \frac12\norm{\nabla \mathcal N(u_0-(u_0)_D)}^2_H
\end{split}\]
for every $t\in[0,T]$. Integrating by parts, rearranging the terms using the Young inequality 
and \eqref{N2}, we deduce that 
\[\begin{split}
  \frac12&\norm{(u_\lambda-m_0)(t)}_*^2 + \int_0^t\norm{\nabla u_\lambda(s)}^2_H\,ds
  +\int_0^t\!\!\int_D\beta_\lambda(u_\lambda(s))u_\lambda(s)\,ds\\
  &\leq \frac12\norm{u_0}_*^2+ \int_0^t\!\!\int_D\nabla u_\lambda(s)\cdot\nabla (B\cdot W)(s)\,ds+
  \int_0^t\!\!\int_D\beta_\lambda(u_\lambda(s))m_0(s)\,ds\\
  &\qquad
  +\frac12\int_0^t\norm{(u_\lambda-m_0)(s)}_H^2\,ds +2C_\pi^2\int_0^t\norm{u_\lambda(s)}^2_H\,ds
  + \norm{g}^2_{L^2(0,T;H)} + 2C_0^2T|D|\\
  &\leq\frac12\norm{u_0}_*^2+
  \frac12\int_0^t\norm{\nabla u_\lambda(s)}^2_H\,ds + \frac12\int_0^t\norm{\nabla(B\cdot W)(s)}^2_H\,ds \\
  &\qquad+\frac1\alpha\int_0^t\!\!\int_Dj^*(\beta_\lambda(u_\lambda(s)))\,ds
  +\int_0^t\!\!\int_Dj(\alpha m_0(s))\,ds
  + \norm{g}^2_{L^2(0,T;H)} + 2C_0^2T|D|\\
  &\qquad+ 4C_\pi^2 \int_0^T\norm{m_0(s)}^2_H\,ds
  +\eps\int_0^t\norm{\nabla (u_\lambda-m_0)(s)}_H^2\,ds + C_\eps\int_0^t\norm{(u_\lambda-m_0)(s)}^2_*\,ds
\end{split}\]
for every $t\in[0,T]$,  $\alpha>1$ and $\eps>0$, for a certain $C_\eps>0$. Now, since 
\[
j^*(\beta_\lambda(u_\lambda))\leq j^*(\beta_\lambda(u_\lambda)) + j((I+\lambda\beta)^{-1}u_\lambda)
=\beta_\lambda(u_\lambda)(I+\lambda\beta)^{-1}u_\lambda\leq\beta_\lambda(u_\lambda)u_\lambda\,,
\]
choosing $\alpha=2$,  $\eps=\frac18$ and the Young inequality yields
\[\begin{split}
  \frac12&\norm{(u_\lambda-m_0)(t)}_*^2 + \frac14\int_0^t\norm{\nabla u_\lambda(s)}^2_H\,ds 
  +\frac12\int_0^t\!\!\int_Dj^*(\beta_\lambda(u_\lambda(s)))\,ds\\
  &\leq\frac12\norm{u_0}_*^2+
  \frac12\norm{\nabla(B\cdot W)}^2_{L^2(0,T; H)} + \int_0^T\!\!\int_Dj(2m_0) + \norm{g}^2_{L^2(0,T; H)} + 2C_0^2T|D|\\
  &\qquad\quad+4C_\pi^2\norm{m_0}^2_{L^2(0,T; H)} + \frac14\norm{\nabla m_0}^2_{L^2(0,T; H)}
  +C\int_0^t\norm{(u_\lambda-m_0)(s)}^2_*\,ds\,.
\end{split}\]
Now, since $\omega\in \Omega_0$, we have that $m_0\in L^2(0,T; V_1)$, $B\cdot W\in L^2(0,T; V_1)$,
as well as
\[
  j(2m_0)=j(2(u_0)_D + 2(B\cdot W))\leq\frac12 j(4(u_0)_D) + \frac12j(4B\cdot W)\,,
\]
where $j(4(u_0)_D)(\omega)<+\infty$ by choice of $\omega\in\Omega_0$ and 
$j(4 B\cdot W)\leq j(4\norm{B\cdot W}_{L^\infty((0,T)\times D)})$ is finite because 
$B\cdot W(\omega) \in L^\infty(0,T; \bar V)\embed L^\infty(0,T; L^\infty(D))$ and $j$
is continuous on $\erre$.
Consequently, for any $\omega\in\Omega_0$, 
the first four terms on the right-hand side are finite. The Gronwall lemma then implies that 
for every $\omega\in\Omega_0$, there is a positive constant $M_\omega>0$ such that 
\beq\label{est1'}
  \norm{j^*(\beta_\lambda(u_\lambda(\omega)))}_{L^1((0,T)\times D)}\leq M_\omega \qquad\forall\,\lambda\in(0,1]\,,
\eeq
hence also, by substitution, that
\beq
  \label{est1''}
  \norm{\beta_\lambda(u_\lambda(\omega))u_\lambda(\omega)}_{L^1((0,T)\times D)}\leq M_\omega
  \qquad\forall\,\lambda\in(0,1]\,.
\eeq

Moreover, testing equation \eqref{app_eq'} by $u_\lambda-B\cdot W$, we get by the Young inequality that 
\[\begin{split}
  \frac12&\norm{(u_\lambda-B\cdot W)(t)}^2_H + \frac12\int_0^t\norm{\Delta u_\lambda(s)}_H^2\,ds+
  \int_0^t\!\!\int_D\beta'_\lambda(u_\lambda(s))|\nabla u_\lambda(s)|^2\,ds\\
  &\leq\frac12\norm{u_0}_H^2+
  \frac12\norm{\Delta(B\cdot W)}^2_{L^2(0,T; H)}+ \int_0^t\!\!\int_D\beta_\lambda(u_\lambda(s))\Delta(B\cdot W)(s)\,ds
  + \frac12\norm{g}^2_{L^2(0,T;H)}\\
  &\quad+(1+C_\pi^2)\int_0^t\norm{u_\lambda(s)-B\cdot W(s)}_H^2\,ds+2C_\pi^2\norm{B\cdot W}^2_{L^2(0,T; H)}
  +2C_0^2T|D|\\
  &\lesssim1+\norm{u_0}_H^2+
  \norm{B\cdot W}^2_{L^2(0,T; V_2)} + \norm{g}^2_{L^2(0,T;H)}\\
  &\quad+\int_0^T\!\!\int_Dj^*(\beta_\lambda(u_\lambda))
  +\int_0^T\!\!\int_Dj(\Delta(B\cdot W))+\int_0^t\norm{u_\lambda(s)-B\cdot W(s)}_H^2\,ds
\end{split}\]
for every $t\in[0,T]$. Recalling that $B\cdot W(\omega) \in L^\infty(0,T; \bar V)\embed L^\infty(0,T; W^{2,\infty}(D))$
implies $j(\Delta(B\cdot W))\in L^1((0,T)\times D)$ by continuity of $j$,
thanks also to the estimate \eqref{est1'} and the Gronwall lemma
we deduce that for every $\omega\in \Omega_0$, there is a positive constant $M_\omega$
such that 
\beq
  \label{est2'}
  \norm{u_\lambda(\omega)}_{L^\infty(0,T; H)}^2 + \norm{u_\lambda(\omega)}^2_{L^2(0,T; V_2)}\leq M_\omega
  \qquad\forall\,\lambda\in(0,1]\,.
\eeq

Finally, going back to \eqref{app_eq'}, 
since $-\Delta:L^1(D)\to \bar V_2^*$ is continuous, 
using \eqref{est1'}--\eqref{est2'} it is not difficult to check that, for $\omega\in\Omega_0$, 
the last two terms on the left-hand side
are bounded in $L^1(0,T; \bar V_2^*)$ uniformly in $\lambda$. By difference we deduce that
for every $\omega\in \Omega_0$, there is a positive constant $M_\omega$
such that 
\beq
  \label{est3'}
  \norm{\partial_t(u_\lambda-B\cdot W)(\omega)}_{L^1(0,T; \bar V_2^*)}\leq M_\omega
  \qquad\forall\,\lambda\in(0,1]\,.
\eeq

\subsection{The passage to the limit}
We pass now to the limit as $\lambda\searrow0$ in the approximated problem \eqref{eq1_app}--\eqref{eq4_app}.

Fix $\omega\in\Omega_0$: by \eqref{est2'},
there is a subsequence $\lambda'=\lambda'(\omega)$ with $\lambda'\to0$ such that
\begin{align}
   \label{conv1}
  u_{\lambda'}(\omega) \wstarto u(\omega) \qquad&\text{in } L^\infty(0,T; H)\,,\\
  \label{conv2}
  u_{\lambda'}(\omega) \wto u(\omega) \qquad&\text{in } L^2(0,T; V_2)\,.
\end{align}
Furthermore,
recalling that $j^*$ is superlinear because $D(\beta)=\erre$, by \eqref{est1'} we infer by the de la Vall\'ee Poussin criterion that 
the family $\{\beta_\lambda(u_\lambda(\omega))\}_\lambda$ is uniformly integrable on $(0,T)\times D$, hence also
relatively weakly compact by the
Dunford-Pettis theorem. We can suppose then with no restriction that 
\beq
  \label{conv3}
  \beta_{\lambda'}(u_{\lambda'}(\omega)) \wto \xi(\omega) \qquad\text{in } L^1((0,T)\times D)\,.
\eeq

Let us show now also a strong convergence for $u_\lambda(\omega)$. 
Since $B\cdot W(\omega)\in L^2(0,T; V_2)$, by \eqref{est2'} we also have that 
$(u_\lambda-B\cdot W)(\omega)$ is bounded in $L^2(0,T; V_2)$;
this fact, together with \eqref{est3'} and the compact inclusion $V_2\cembed V_1$, 
implies by the classical Aubin-Lions theorems on compactness (see \cite[Cor.~4, p.~85]{simon}) that the family
$\{(u_\lambda-B\cdot W)(\omega)\}_\lambda$ is relatively compact in $L^2(0,T; V_1)$. Since $B\cdot W$ is fixed, we can 
assume with no restriction also that 
\beq
  \label{conv4}
  u_{\lambda'}(\omega)\to u(\omega) \qquad\text{in } L^2(0,T; V_1)\,.
\eeq

Take now a test function $\varphi\in \bar V_2$ and fix $\omega\in \Omega_0$:
we will not write 
explicitly the dependence on $\omega$
in this section.
For every $t\in[0,T]$,
by \eqref{conv1}--\eqref{conv4}, using the definition of $A_\lambda$ and the Lipschitzianity of $\pi$ it is immediate to check that
\[\begin{split}
  \int_0^t&\sp{A_{\lambda'}(s,u_{\lambda'}(s))}{\varphi}_{V_2}\,ds\\
  &=\int_0^t\!\!\int_D\Delta u_{\lambda'}(s)\Delta\varphi\,ds - 
  \int_0^t\!\!\int_D\left(\beta_{\lambda'}(u_{\lambda'}(s)) + \pi(u_{\lambda'}(s)) + g(s)\right)\Delta\varphi\,ds\\
  &\to\int_0^t\!\!\int_D\Delta u(s)\Delta\varphi\,ds - 
  \int_0^t\!\!\int_D\left(\xi(s) + \pi(u(s)) + g(s)\right)\Delta\varphi\,ds\,.
\end{split}\]
By difference in \eqref{app_eq}, we have that $u_{\lambda'}(t)$ converges weakly star in $\bar V_2^*$ for every $t\in[0,T]$: 
since $u_\lambda(\omega)$ is bounded in $L^\infty(0,T;H)$ we deduce that
\[
  u_{\lambda'}(t) \wto u(t) \qquad\text{in } H\,, \quad\forall\, t\in[0,T]\,.
\]
Consequently, letting $\lambda'\searrow0$ in \eqref{app_eq} we have that 
\[
  u(t) + \int_0^t\Delta^2u(s)\,ds - \int_0^t\Delta\xi(s)\,ds - \int_0^t\Delta\pi(u(s))\,ds - \int_0^t\Delta g(s)\,ds
  =u_0 + \int_0^tB(s)\,dW_s
\]
for every $t\in[0,T]$, in $\Omega_0$. Since $\P(\Omega_0)=1$, the previous relation holds $\P$-almost surely.

Furthermore, by the two convergences \eqref{conv3}--\eqref{conv4}, we can apply a generalized result on the strong-weak
closure of maximal monotone graphs due to Br\'ezis (see \cite[Thm.~18, p.~126]{brezis-monot}) to infer that 
\[
  \xi(\omega) \in \beta(u(\omega)) \quad\text{a.e.~in } (0,T)\times D\,, \quad\forall\,\omega\in\Omega_0\,.
\]
Moreover, by the weak lower semicontinuity of the convex integrands and the fact that
$j((I+\lambda\beta)^{-1}x)+j^*(\beta_\lambda(x))=(I+\lambda\beta)^{-1}x\beta_\lambda(x)$ for every $x\in\erre$, 
for $\omega\in\Omega_0$
we have by \eqref{est1''} that
\[\begin{split}
  \int_0^T\!\!\int_D\left(j(u)+ j^*(\xi)\right) &\leq\liminf_{\lambda'\searrow0}\int_0^T\!\!\int_D\left(j((I+\lambda'\beta)^{-1}u_{\lambda'}) + 
  j^*(\beta_{\lambda'}(u_{\lambda'}))\right)\\
  &= \liminf_{\lambda'\searrow0}\int_0^T\!\!\int_D\beta_{\lambda'}(u_{\lambda'})(I+\lambda'\beta)^{-1}u_{\lambda'}\leq
  \int_0^T\!\!\int_D\beta_{\lambda'}(u_{\lambda'})u_{\lambda'}\leq M_\omega\,,
\end{split}\]
so that $j(u) + j^*(\xi) \in L^1((0,T)\times D)$ $\P$-almost surely.

Finally, it is clear from the limit equation and the regularities of $u$ and $\xi$ that the 
trajectories of $u$ are in $C^0([0,T]; V_2^*)$ $\P$-almost surely; since 
we know that $u \in L^\infty(0,T; H)$, it follows that 
$u \in C^0([0,T]; V_1^*)$ $\P$-almost surely as well.

\subsection{Measurability and integrability of the limit processes}
\label{meas}
The processes $u$ and $\xi$ constructed above not not have any obvious measurability property in $\Omega$,
because of the way they have been built: indeed, the subsequence $\lambda'$ depends on $\omega$ (not in
a measurable way  in general). We prove here that $u$ and $\xi$ are actually measurable in some sense
and suitably integrable on $\Omega$ as well.

The first step in to prove a pathwise uniqueness result for the problem: in other words, we prove that the two processes 
$u$ and $\xi-\xi_D$ are unique whenever $(u,\xi)$ solves the limit equation.
More in detail, let $(u_i,\xi_i)$, $i=1,2$, be two pair of processes such that, for $i=1,2$,
\begin{gather*}
u_i\in L^\infty(0,T; H)\cap L^2(0,T; V_2)\,, \qquad \xi \in L^1((0,T)\times D)\,,\\
\xi_i\in\beta(u_i) \quad\text{a.e.~in } (0,T)\times D\,, \qquad j(u_i)+j^*(\xi_i)\in L^1((0,T)\times D)
\end{gather*}
$\P$-almost surely, and
\[
u_i(t) + \int_0^t\Delta^2u_i(s)\,ds - \int_0^t\Delta\left(\xi_i(s)+\pi(u_i(s))+g(s)\right)\,ds
=u_0 + \int_0^tB(s)\,dW_s
\]
for every $t\in[0,T]$, $\P$-almost surely: we show that, setting $\bar u:=u_1-u_2$ and $\bar\xi:=\xi_1-\xi_2$, then
$\bar u=0$ and $\bar\xi-\bar\xi_D=0$. First of all, by difference in the limit equations we have
\[
\bar u(t) + \int_0^t\Delta^2\bar u(s)\,ds - \int_0^t\Delta\bar\xi(s)\,ds - \int_0^t\Delta(\pi(u_1(s))-\pi(u_2(s)))\,ds=0
\quad\forall\,t\in[0,T]\,, \quad\P\text{-a.s.}
\]
We would like to proceed now as in Subsection~\ref{second_est}: 
however, the resolvent $(I+\delta\Delta^2)^{-1}$
is not a contraction on $L^1$, and this may cause some problems 
when letting $\delta\searrow0$ in the term with $\bar\xi$.
Consequently,
we apply for any $\delta\in(0,1)$
the operator $(I-\delta\Delta)^{-k}$ which contracts also in $L^1$,
commutes with $-\Delta$ and $\Delta^2$, and maps $L^1(D)$ into $V_2$
for $k$ sufficiently large (being fixed) by the Sobolev embeddings theorems. 
Consequently, using the superscript $\delta$ to denote
the action of $(I-\delta\Delta)^{-k}$, we deduce
\[
  \bar u^\delta(t)+ \int_0^t\Delta^2\bar u^\delta(s)\,ds - \int_0^t\Delta\bar\xi^\delta(s)\,ds 
  - \int_0^t\Delta(\pi(u_1(s))-\pi(u_2(s)))^\delta\,ds=0
\]
for every $t\in[0,T]$, $\P$-almost surely. Fix now $\omega$: 
the previous equation can be written as
\[
\partial_t \bar u^\delta + \Delta^2\bar u^\delta -\Delta\bar\xi^\delta -\Delta(\pi(u_1)-\pi(u_2))^\delta=0 \quad\text{in } H\,,
\quad\text{for a.e.~}t\in(0,T)\,.
\]
Testing by $\frac{1}{|D|}$, it is clear that $(\bar u)_D=0$, and hence also $(\bar u^\delta)_D=0$ since $(I-\delta\Delta)^{-1}$ is 
mean-preserving. Consequently, we can test the previous equation by $\mathcal N\bar u^\delta$: recalling
\eqref{N2}--\eqref{N3} we infer that
\[\begin{split}
  \frac12\norm{\nabla\mathcal N\bar u^\delta(t)}_H^2 &+
  \int_0^t\norm{\nabla  \bar u^\delta(s)}^2_H\,ds+
  \int_0^t\!\!\int_D\bar\xi^\delta(s)\bar u^\delta(s)\,ds=\\
  &-\int_0^t\!\!\int_D\left(\pi(u_1(s))-\pi(u_2(s))\right)^\delta\bar u^\delta(s)\,ds
  \leq  C_\pi\int_0^t\norm{\bar u(s)}_H\norm{\bar u^\delta(s)}_H\,ds\\
  &\leq\int_0^t\norm{\bar u(s)}_H^2\,ds + 
  \eps\int_0^t\norm{\nabla\bar u^\delta(s)}^2_H\,ds + C_\eps\int_0^t\norm{\nabla\mathcal N\bar u^\delta(s)}_H^2\,ds
\end{split}\]
for any $\eps>0$ and a a positive constant $C_\eps$.
Choosing $\eps$ small enough, rearranging the terms and using the Gronwall lemma yields
\[
  \norm{\nabla\mathcal N\bar u^\delta(t)}_H^2 +\int_0^t\norm{\nabla  \bar u^\delta(s)}^2_H\,ds +
  \int_0^t\!\!\int_D\bar\xi^\delta(s)\bar u^\delta(s)\,ds \lesssim \int_0^t\norm{\bar u(s)}_H^2\,ds
  \qquad\forall\,t\in[0,T]\,.
\]
We want to let $\delta\searrow0$. Since $\bar u(t)\in V_1^*$ and $\bar u\in L^2(0,t; V_2)$, 
by the contraction properties of the resolvent operator on the spaces $V_1^*$ and $V_1$
it is clear that $\norm{\nabla\mathcal N\bar u^\delta(t)}_H^2\to\norm{\nabla\mathcal N\bar u(t)}_H^2$ 
and $\int_0^t\norm{\nabla \bar u^\delta(s)}_H^2\,ds \to
\int_0^t\norm{\nabla \bar u(s)}_H^2\,ds$.
Moreover, since $\bar \xi^\delta\to\xi$ in $L^1((0,T)\times D)$ and $\bar u^\delta\to\bar u$ in $L^2(0,T; H)$, we can assume
with no restrictions that $\bar \xi^\delta\to\bar\xi$ and $\bar u^\delta\to\bar u$ almost everywhere in $(0,T)\times D$.
We show that $\bar\xi^\delta \bar u^\delta\to \bar\xi \bar u$ in $L^1((0,T)\times D)$ using Vitali's convergence theorem, i.e.~proving
that the family $\{\bar\xi^\delta \bar u^\delta\}_\delta$ is uniformly integrable on $(0,T)\times D$. To this end,
let us recall that by the symmetry of $j$ an easy computation implies that 
\[
  \forall\, v\in L^1(D):\; j^*(v) \in L^1(D) \qquad\exists\,\eta>0:\; j^*(\eta|v|) \in L^1(D)\,.
\]
Since $j^*(\xi_2) \in L^1((0,T)\times D)$, let $\eta\in(0,1)$ such that $j^*(\eta|\xi_2|) \in L^1((0,T)\times D)$.
By the generalized Young inequality, the symmetry of $j$ and the generalized Jensen inequality for positive operators
(see \cite{haase}) we have
\[\begin{split}
  \pm\frac\eta4\bar\xi^\delta \bar u^\delta&\leq j\left(\pm\frac12\bar u^\delta \right) + j^*\left(\frac\eta2\bar\xi^\delta\right)\lesssim
  \frac12\left(j(u_1^\delta) + j(u_2^\delta) + j^*(\xi_1^\delta) + j^*(-\eta\xi_2^\delta)\right)\\
  &\leq\frac12(I-\delta\Delta)^{-k}\left(j(u_1) + j(u_2) + j^*(\xi_1) + j^*(\eta|\xi_2|)\right)\,.
\end{split}\]
Since the sum of the four terms in brackets on the right-hand side is in $L^1((0,T)\times D)$, by the contraction properties of 
$(I-\delta\Delta)^{-k}$ on $L^1(D)$ we have that the right-hand side converges in $L^1((0,T)\times D)$.
This implies that $\{\bar\xi^\delta \bar u^\delta\}_\delta$ is uniformly integrable on $(0,T)\times D$,
hence by Vitali's theorem that $\bar\xi^\delta \bar u^\delta\to \bar\xi \bar u$ in $L^1((0,T)\times D)$.
Letting then $\delta\searrow0$, we deduce again by \eqref{N2} that
\[\begin{split}
   \norm{\nabla \mathcal N\bar u(t)}_H^2 &+\int_0^t\norm{\nabla  \bar u(s)}^2_H\,ds 
   + \int_0^t\!\!\int_D\bar\xi(s)\bar u(s)\,ds \\
   &\lesssim
   \eps\int_0^t\norm{\nabla \bar u(s)}_H^2\,ds + C_\eps\int_0^t\norm{\nabla \mathcal N\bar u(s)}_H^2\,ds
  \qquad\forall\,t\in[0,T]\,
\end{split}\]
for every $\eps>0$ and a certain $C_\eps>0$.
Choosing $\eps$ small enough, by the monotonicity of $\beta$ and the Gronwall lemma
this readily implies that $\bar u=0$. Furthermore, by substitution in the equation itself,
we have that 
\[
  -\int_0^t\Delta\bar\xi(s)\,ds = 0 \qquad\forall\,t\in[0,T]\,,
\]
which implies that $-\Delta\bar\xi=0$. Hence, 
we have that $\bar\xi$ is constant in $D$, or, equivalently,
$\bar\xi-\bar\xi_D=0$. 
This proves that the processes $u$ and $\xi-\xi_D$ constructed path-by-path in the previous section
are indeed unique in the sense specified above.

The second step to prove measurability is to show that actually $\lambda'$ does not depend on $\omega\in\Omega_0$.
Indeed, from any subsequence of $\lambda$ we can extract a further subsequence $\lambda'$,
depending on $\omega$, such that all the convergences \eqref{conv1}--\eqref{conv4} hold,
hence also in particular $\beta_{\lambda'}(u_\lambda')-(\beta_{\lambda'}(u_\lambda'))_D$ 
converges weakly in $L^1((0,T)\times D)$.
Since we have just proved that the processes $u$ and $\xi-\xi_D$ are unique, 
by an elementary result of classical analysis we deduce that the convergences \eqref{conv1}--\eqref{conv2}
and \eqref{conv4} hold along the entire sequence $\lambda$, as well as
\[
\beta_{\lambda'}(u_\lambda')-(\beta_{\lambda}(u_\lambda))_D \wto \xi-\xi_D \qquad\text{in } L^1((0,T)\times D)
\qquad\P\text{-a.s.}\,.
\]

Let us show that the measurability of $u$ and $\xi$.
From the strong convergence
\eqref{conv4} we know that $u_\lambda\to u$ strongly in $L^2(0,T; H)$ $\P$-almost surely,
which implies for a subsequence the convergence
$\P\otimes dt$-almost everywhere to the same limit. 
Since $u_\lambda$ is adapted with continuous trajectories, this ensures that $u$ is predictable in $H$.
Let us focus now on $\xi$: since we only have a weak convergence, a different argument is needed.
First, we set
$\xi_\lambda:=\beta_\lambda(u_\lambda)$ and
define for any $h \in L^\infty((0,T) \times D)$
\[
F_\lambda:= \int_0^T\!\!\int_D (\xi_\lambda-(\xi_\lambda)_D) h \,,\qquad 
F:= \int_0^T\!\! \int_D (\xi-\xi_D) h\,.
\]
Since $\xi_\lambda-(\xi_\lambda)_D\wto\xi-\xi_D$ 
in $L^1((0,T)\times D)$, we have that
$F_\lambda\to F$  $\P$-almost surely in $\Omega$.
Let us prove that actually $F_\lambda\wto F$ in $L^1(\Omega)$.
For any $\ell\in L^\infty(\Omega)$, setting $j_0(\cdot):= j^*(\frac\cdot{2M})$ and 
$M= 1/ [(\norm h_{L^\infty((0,T) \times D)} \vee 1)(\norm\ell_{L^\infty( \Omega)} \vee 1)]$,
by Jensen's inequality we get 
\begin{equation*}
\E j_0(F_\lambda \ell) = \E j_0 \left( \int_0^T\!\!\int_D (\xi_\lambda -(\xi_\lambda)_D)h\ell\right)
\lesssim \E \int_0^T\!\!\int_D j^*(\xi_\lambda)\, dz ds\,,
\end{equation*}
where the last term is bounded thanks to \eqref{est2}. Since $j_0$ is superlinear, by the de la Vall\'ee Poussin criterion 
the sequence $F_\lambda \ell$ is uniformly integrable in $\Omega$, which implies the strong convergence
$F_\lambda \ell\to F \ell$ in $L^1(\Omega)$ thanks to Vitali's theorem. Since $\ell$ is arbitrary,
we get the weak convergence in $L^1(\Omega)$ of $(F_\lambda)_\lambda$: hence, we have that 
$\xi_\lambda-(\xi_\lambda)_D\wto \xi-\xi_D$  
in $L^1(\Omega \times (0,T) \times D)$. 
Taking into account that $\{\xi_\lambda\}_\lambda$ is weakly relatively compact in $L^1(\Omega\times(0,T)\times D)$
by the estimate \eqref{est2}, it is not restrictive to assume that 
$\xi_\lambda \wto \xi$ in $L^1(\Omega\times(0,T)\times D)$.
By Mazur's lemma, there is a subsequence (independent of $\omega$)
made up of convex combinations
of $\{\xi_\lambda\}_\lambda$ 
which converge strongly to $\xi$ in $L^1(\Omega\times(0,T)\times D)$.
Since $\xi_\lambda$ are predictable 
(as finite convex combination of predictable processes), the limit $\xi$ is a predictable $L^1(D)$-valued process.
Similarly, one can show that $u$ is also a measurable $V_2$-valued process.

Finally, we prove some integrability properties for the limit processes $(u,\xi)$. 
By the weak lower semicontinuity of the norms and the estimates \eqref{est1}--\eqref{est2} we have
\begin{align}
  \label{integ1}
  \E\norm{u}_{L^\infty(0,T; H)}^2 &\leq \liminf_{\lambda\searrow0}\E\norm{u_\lambda}^2_{L^\infty(0,T; H)}\leq M\,,\\
  \label{integ2}
  \E\norm{u}_{L^2(0,T; V_2)}^2 &\leq \liminf_{\lambda\searrow0}\E\norm{u_\lambda}^2_{L^2(0,T; V_2)}\leq M\,,\\
  \label{integ3}
  \E\norm{\xi}_{L^1((0,T)\times D)} &\leq\liminf_{\lambda\searrow0}\E\norm{\beta_\lambda(u_\lambda)}_{L^1(0,T)\times D}\leq M\,,
\end{align}
while the weak lower semicontinuity of the convex integrals yields
\beq\label{integ4}
\begin{split}
  \E\int_0^T\!\!\int_D\left(j(u)+j^*(\xi)\right)&\leq \liminf_{\lambda\searrow0}\E\int_0^T\!\!\int_D\left(j((I+\lambda\beta)^{-1}u_\lambda)
  +j^*(\beta_\lambda(u_\lambda))\right)\\
  &=\E\int_0^T\!\!\int_D(I+\lambda\beta)^{-1}u_\lambda\beta_\lambda(u_\lambda)\leq
  \E\int_0^T\!\!\int_D\beta_\lambda(u_\lambda)u_\lambda\leq M\,.
\end{split}\eeq
This concludes the proof of Theorem~\ref{thm1} under the additional assumption \eqref{strong_B}.

\subsection{Continuous dependence}
\label{cont_dep}
We prove here the continuous dependence result contained in Theorem~\ref{thm2},
which will be needed in order to remove the hypothesis 
\eqref{strong_B} in the next section.

In the notation and assumptions of Theorem~\ref{thm2}, 
setting by convenience $u:=u_1-u_2$, $\xi:=\xi_1-\xi_2$, $P:=\pi(u_1)-\pi(u_2)$, $u_0:=u_0^1-u_0^2$,
$g:=g_1-g_2$ and $B:=B_1-B_2$ we have
\[
  u(t) + \int_0^t\Delta^2u(s)\,ds - \int_0^t\Delta\left(\xi(s)+P(s)\right)\,ds
  =u_0 +\int_0^t\Delta g(s)\,ds + \int_0^tB(s)\,dW_s
\]
for every $t\in[0,T]$, $\P$-almost surely.
Using a superscript $\delta$ to denote the action of the resolvent $(I-\delta\Delta)^{-k}$ for any $\delta\in(0,1)$, as 
in the previous section,
and recalling that $-\Delta$, $\Delta^2$ and $(I-\delta\Delta)^{-k}$ commute, we have
\[
  u^\delta(t) + \int_0^t\Delta^2u^\delta(s)\,ds - \int_0^t\Delta\left(\xi^\delta(s)+P^\delta(s)\right)\,ds
  =u^\delta_0 +\int_0^t\Delta g^\delta(s)\,ds+ \int_0^tB^\delta(s)\,dW_s
\]
for every $t\in[0,T]$, $\P$-almost surely.
Now, the hypothesis on the mean of $(u_0^i)_D+(B_i\cdot W)_D$ in Theorem~\ref{thm2}
implies that $u_D=0$, hence also 
$(u^\delta)_D=0$. Consequently, 
It\^o's formula for the function $\Phi(u^\delta)$ yields
\[
  \begin{split}
  \frac12\norm{u^\delta(t)}_*^2
  &+\int_0^t\norm{\nabla u^\delta(s)}_H^2\,ds
  +\int_0^t\!\!\int_D\xi^\delta(s)u^\delta(s)\,ds\\
  &= \frac12\norm{u^\delta_0}_*^2 
  - \int_0^t\!\!\int_DP^\delta(s)u^\delta(s)\,ds
  - \int_0^t\!\!\int_Dg^\delta(s)u^\delta(s)\,ds\\
  &\quad+\frac12\int_0^t\operatorname{Tr}\left[(B^\delta(s))^*\mathcal N(\cdot-(\cdot)_D)B^\delta(s)\right]\,ds
  +\int_0^t\left(\mathcal Nu_\delta(s), B^\delta(s)\right)_H\,dW_s
  \end{split}
\]
for every $t\in[0,T]$, $\P$-almost surely. Proceeding now as in Subsection~\ref{second_est},
if $\{e_k\}_{k\in\enne}$ is an orthonormal system of $U$, 
we have
\[\begin{split}
\operatorname{Tr}\left[(B^\delta)^*\mathcal N(\cdot-(\cdot)_D)B^\delta\right]&=
\sum_{k=0}^\infty\left(\mathcal N(B^\delta e_k-(B^\delta e_k)_D), B^\delta e_k\right)_H\\
&=\sum_{k=0}^\infty\norm{\nabla\mathcal N(B^\delta e_k-(B^\delta e_k)_D)}^2_H
\leq\sum_{k=0}^\infty\norm{B^\delta e_k}^2_{*}\lesssim\norm{B}^2_{\cL_2(U,V_1^*)}
\end{split}\]
and
\[\begin{split}
\E&\sup_{t\in[0,T]}\abs{\int_0^t\left(\mathcal Nu^\delta(s), B^\delta(s)\right)_H\,dW_s}
\lesssim
\E\left(\int_0^T\norm{\mathcal Nu^\delta(s)}_{V_1}^2\norm{B^\delta(s)}_{\cL_2(U,V_1^*)}^2\,ds\right)^{1/2}\\
&\lesssim\eps\E\sup_{t\in[0,T]}\norm{\nabla\mathcal Nu^\delta(t)}_H^2 + 
\frac1{4\eps}\norm{B}^2_{L^2(\Omega; L^2(0,T; \cL_2(U,V_1^*)))}
\end{split}\]
for every $\eps>0$.
Choosing $\eps$ small enough, rearranging the terms, taking supremum in time and expectations
and using the Lipschitz continuity of $\pi$ together with the Young inequality and \eqref{N2} we have
\[\begin{split}
  &\norm{u^\delta}^2_{L^2(\Omega; L^\infty(0,t; V_1^*))} + \norm{\nabla u^\delta}^2_{L^2(\Omega; L^2(0,T; H))}
  +\E\int_0^T\!\!\int_D\xi^\delta u^\delta\\
  &\lesssim\norm{u_0}^2_{L^2(\Omega; V_1^*)} + \norm{B}^2_{L^2(\Omega; L^2(0,T; \cL_2(U,V_1^*)))}\\
  &\quad+C_\pi\E\int_0^t\norm{u(s)}_H\norm{u^\delta(s)}_H\,ds + \E\int_0^t\norm{g^\delta(s)}_*\norm{u^\delta(s)}_1\,ds\\
  &\lesssim\norm{u_0}^2_{L^2(\Omega; V_1^*)} + \norm{B}^2_{L^2(\Omega; L^2(0,T; \cL_2(U,V_1^*)))}
  +\norm{g}^2_{L^2(\Omega; L^2(0,T; V_1^*))}\\
  &\quad+\int_0^t\norm{u(s)}_H^2\,ds+\eta\E\int_0^T\norm{\nabla u^\delta(s)}^2_H\,ds + 
  C_\eta\int_0^t\norm{u^\delta}_{L^2(\Omega; L^\infty(0,s;V_1^*))}\,ds
\end{split}\]
for every $\eta>0$ and $t\in[0,T]$. Hence, the Gronwall lemma yields
\[\begin{split}
  &\norm{u^\delta}^2_{L^2(\Omega; L^\infty(0,t; V_1^*))} + \norm{\nabla u^\delta}^2_{L^2(\Omega; L^2(0,T; H))}
  +\E\int_0^t\!\!\int_D\xi^\delta u^\delta\\
  &\lesssim\norm{u_0}^2_{L^2(\Omega; V_1^*)}+\norm{g}^2_{L^2(\Omega; L^2(0,T; V_1^*))}+ 
  \norm{B}^2_{L^2(\Omega; L^2(0,T; \cL_2(U,V_1^*)))}+\int_0^t\norm{u(s)}_H^2\,ds
\end{split}\]
for every $t\in[0,T]$.
We need to let $\delta\searrow0$. To this end, note that
if $k=2$ then from the definition of $(I-\delta)^{-k}$ we have 
$u^\delta-2\delta\Delta u^\delta+\delta^2\Delta^2u^\delta=u$:
testing by $u^\delta$ it easily follows from the Young inequality that 
\[
  \frac12\norm{u^\delta}^2_H +2\delta\norm{\nabla u^\delta}_H^2 + \delta^2\norm{\Delta u^\delta}_H^2\leq \frac12\norm{u}_H^2
  \qquad\forall\,\delta\in(0,1)\,,
\]
while testing by $-\Delta u^\delta$ we have
\[
  \norm{\nabla u^\delta}_H^2 + 2\delta\norm{\Delta u^\delta}_H^2 + \delta^2\norm{\nabla\Delta u^\delta}_H^2
  =\int_Du(-\Delta u^\delta)\leq \delta\norm{\Delta u^\delta}_H^2 + \frac1{4\delta}\norm{u}_H^2\,,
\]
so that, in particular, we have $2\delta\norm{\nabla u^\delta}_H^2\leq\frac12\norm{u}_H^2$ and
$\delta^3\norm{\nabla\Delta u^\delta}_H^2\leq \frac14\norm{u}_H^2$.
Hence, bearing in mind these considerations, we have that 
\[\begin{split}
  \norm{u-u^\delta}_*^2&=\norm{\nabla\mathcal N(u-u^\delta)}^2_H\leq
  8\delta^2\norm{\nabla u^\delta}_H^2+2\delta^4\norm{\nabla\Delta u^\delta}_H^2\\
  &\leq 2\delta\norm{u}_H^2 + \frac\delta2\norm{u}_H^2=\frac52\delta\norm{u}_H^2\,.
\end{split}\]
In general, if $k\geq2$, then the previous inequality still holds since $(I-\delta\Delta)^{-1}$
contracts on $V_1^*$.
Since $u\in L^2(\Omega; L^\infty(0,T; H))$, this ensures that $u^\delta\to u$ in $L^2(\Omega; L^\infty(0,T; V_1^*))$\,.
Moreover, since $u\in L^2(\Omega; L^2(0,T; V_2))$, 
it is clear that $\nabla u^\delta\to\nabla u$ in $L^2(\Omega; L^2(0,T; H))$.
Finally, proceeding exactly as in the previous subsection, one can show that the family $\{\xi^\delta u^\delta\}_\delta$
is uniformly integrable on $\Omega\times(0,T)\times D$, hence by Vitali's convergence theorem 
that $\xi^\delta u^\delta\to\xi u$ in $L^1(\Omega\times(0,T)\times D)$.
Letting then $\delta\searrow0$ with this information, by \eqref{N2} we have
\[\begin{split}
  \norm{u}^2_{L^2(\Omega; L^\infty(0,t; V_1^*))} &+ \norm{\nabla u}^2_{L^2(\Omega; L^2(0,T; H))}
  +\E\int_0^T\!\!\int_D\xi u\\
  &\lesssim\norm{u_0}^2_{L^2(\Omega; V_1^*)}
  +\norm{g}^2_{L^2(\Omega; L^2(0,T; V_1^*))}+ 
  \norm{B}^2_{L^2(\Omega; L^2(0,T; \cL_2(U,V_1^*)))}\\
  &+\eta\int_0^T\norm{\nabla u(s)}_H^2+
  C_\eta\int_0^t\norm{\nabla\mathcal{N}u(s)}_H^2\,ds
\end{split}\]
for every $\eta>0$ and a certain $C_\eta>0$; hence, 
the inequality of Theorem~\ref{thm2} is proved choosing $\eta$ small enough
using the Gronwall lemma and the monotonicity of $\beta$.

Moreover, it is clear from the continuous dependence that if $u_0^1=u_0^2$, $g_1=g_2$ and $B_1=B_2$, 
then $u_1=u_2$, hence also by substitution in the equation $-\Delta(\xi_1-\xi_2)=0$: this implies that $\xi_1-\xi_2$
is constant on $D$, hence necessarily equal to $(\xi_1-\xi_2)_D$.
Finally, if $\beta$ is also single-valued, then $\xi_1=\xi_2$ follows directly from the fact that $u_1=u_2$
and $\xi_i\in\beta(u_i)$ for $i=1,2$. This completes the proof of Theorem~\ref{thm2}.

\subsection{Conclusion of the proof}
Our goal is now to remove the assumption \eqref{strong_B} and to conclude
the proof of Theorem~\ref{thm1}.

Let $B \in L^2(\Omega; L^2(0,T; \cL_2(U,H)))$ and consider, for any $\eps\in(0,1)$, the operator
\[
  B^\eps:=(I-\eps\Delta)^{-2}B \in L^2(\Omega; L^2(0,T; \cL_2(U, V_4)))\,.
\]
Since $N\in\{2,3\}$, by the Sobolev embeddings we have that $V_4\embed W^{2,\infty}(D)\cap V_2$,
hence we can apply the results already proved in the previous sections and infer that the problem
\eqref{eq1}--\eqref{eq4} with additive noise $B^\eps$ and initial datum $u_0$ admits a unique strong solution $(u^\eps, \xi^\eps)$
for every $\eps\in(0,1)$.

Now, taking into account the inequalities \eqref{integ1}--\eqref{integ4} and recalling how
the constant $M$ was chosen in the estimates \eqref{est1}--\eqref{est2},
we deduce that 
\[\begin{split}
  &\norm{u^\eps}^2_{L^2(\Omega; L^\infty(0,T; H))} + \norm{u^\eps}^2_{L^2(\Omega; L^2(0,T; V_2))}
  +\norm{j(u^\eps)}_{L^1(\Omega\times(0,T)\times D)}
  +\norm{j^*(\xi^\eps)}_{L^1(\Omega\times(0,T)\times D)}\\
  &\lesssim 1+ \norm{u_0}^2_{L^2(\Omega; H)}
  +\norm{B^\eps}^2_{L^2(\Omega; L^2(0,T; \cL_2(U,H)))} + \norm{g}^2_{L^2(\Omega; L^2(0,T; H))}+
  \norm{j(\alpha m^\eps)}_{L^1(\Omega\times(0,T))}\,,
\end{split}\]
for a sufficiently large $\alpha >0$ fixed (independent of $\eps$), where $m^\eps:=(u_0)_D + (B^\eps\cdot W)_D$.
By the contraction properties of the resolvent and the Jensen inequality for positive operators
(see \cite{haase}) we have
\[
  \norm{B^\eps}^2_{L^2(\Omega; L^2(0,T; \cL_2(U,H)))}
  \leq\norm{B}^2_{L^2(\Omega; L^2(0,T; \cL_2(U,H)))}
\]
and
\[\begin{split}
  \norm{j(\alpha m^\eps)}_{L^1(\Omega\times(0,T))}&\leq \frac12\norm{j(2\alpha(u_0)_D)}_{L^1(\Omega\times(0,T))} 
  + \frac12\norm{j(2\alpha(B^\eps\cdot W)_D)}_{L^1(\Omega\times(0,T))}\\
  &\leq\frac T2\norm{j(2\alpha(u_0)_D)}_{L^1(\Omega)} + \frac12\norm{j(2\alpha(B\cdot W)_D)}_{L^1(\Omega\times(0,T))} 
\end{split}\]
for every $\eps\in(0,1)$, where the terms on the right-hand sides 
are finite by the assumptions \eqref{u0} and \eqref{add1}--\eqref{add2}. 
We deduce that there is a positive constant $M$, independent of $\eps$, such that 
the following estimate holds for every $\eps\in(0,1)$:
\[
  \norm{u^\eps}^2_{L^2(\Omega; L^\infty(0,T; H))} + \norm{u^\eps}^2_{L^2(\Omega; L^2(0,T; V_2))}
  +\norm{j(u^\eps)}_{L^1(\Omega\times(0,T)\times D)}
  +\norm{j^*(\xi^\eps)}_{L^1(\Omega\times(0,T)\times D)}\leq M\,.
\]

Furthermore, by Theorem~\ref{thm2}
we have also that 
\[
  \norm{u^\eps-u^\delta}^2_{L^2(\Omega; C^0([0,T]; V_1^*))}
  +\norm{\nabla(u^\eps-u^\delta)}^2_{L^2(\Omega; L^2(0,T; H))}\lesssim
  \norm{B^\eps-B^\delta}^2_{L^2(\Omega; L^2(0,T; \cL_2(U,V_1^*)))}
\]
for every $\eps,\delta\in(0,1)$. Moreover, 
noting that $(u^\eps-u^\delta)_D=((B^\eps-B^\delta)\cdot W)_D$,
by the inequality \eqref{N2} and the Burkholder-Davis-Gundy inequality we have that 
\[\begin{split}
  &\norm{u^\eps-u^\delta}^2_{L^2(\Omega; L^2(0,T; H))}\\
  &\lesssim
  \norm{(u^\eps-u^\delta)-(u^\eps-u^\delta)_D}^2_{L^2(\Omega; L^2(0,T; H))}+
  \norm{((B^\eps-B^\delta)\cdot W)_D}^2_{L^2(\Omega; L^2(0,T; H))}\\
  &\lesssim\norm{\nabla((u^\eps-u^\delta)-(u^\eps-u^\delta)_D)}^2_{L^2(\Omega; L^2(0,T; H))}
  +\norm{(u^\eps-u^\delta)-(u^\eps-u^\delta)_D}^2_{L^2(\Omega; L^2(0,T; V_1^*))}\\
  &\qquad\qquad+\norm{B^\eps-B^\delta}^2_{L^2(\Omega; L^2(0,T; \cL_2(U,H)))}\\
  &\lesssim\norm{\nabla(u^\eps-u^\delta)}^2_{L^2(\Omega; L^2(0,T; H))}+
  \norm{u^\eps-u^\delta}^2_{L^2(\Omega; L^2(0,T; V_1^*))}+\norm{B^\eps-B^\delta}^2_{L^2(\Omega; L^2(0,T; \cL_2(U,H)))}
\end{split}\]
from which infer then that 
\[
  \norm{u^\eps-u^\delta}^2_{L^2(\Omega; L^2(0,T; H))}\lesssim
  \norm{B^\eps-B^\delta}^2_{L^2(\Omega; L^2(0,T; \cL_2(U,H)))} \qquad\forall\,\eps,\delta\in(0,1)\,.
\]
Since $B^\eps\to B$ in $L^2(\Omega; L^2(0,T; \cL_2(U,H)))$ as $\eps\searrow0$
this implies that $(u^\eps)_\eps$ is Cauchy in the space
$L^2(\Omega; L^2(0,T; V_1))\cap L^2(\Omega; C^0([0,T]; V_1^*))$.

Taking into account these consideration, recalling that $j^*$ is superlinear, 
by the de la Vall\'ee Poussin and Dunford-Pettis theorems we deduce that there is a pair $(u,\xi)$ such that,
along a subsequence (that we still denote by $\eps$ for simplicity),
\begin{align*}
  u^\eps \to u \qquad&\text{in } L^2(\Omega; L^2(0,T; V_1))\cap L^2(\Omega; C^0([0,T]; V_1^*))\,,\\
  u^\eps \wto u \qquad&\text{in } L^2(\Omega; L^2(0,T; V_2))\,,\\
  \xi^\eps \wto \xi \qquad&\text{in } L^1(\Omega\times(0,T)\times D)\,.
\end{align*}
Using the lower semicontinuity of the norms and the fact that $(u^\eps)_\eps$
and $j(u^\eps)+j^*(\xi^\eps)$ are bounded in
$L^2(\Omega; L^\infty(0,T; H))$ and $L^1(\Omega\times(0,T)\times D)$, respectively,
it is a standard matter to check that 
\[
  u \in L^2(\Omega; L^\infty(0,T; H))\,, \qquad j(u)+j^*(\xi) \in L^1(\Omega\times(0,T)\times D)\,.
\]
Furthermore, the strong convergence of $u^\eps$ to $u$ allows to apply the 
generalized result on the strong-weak closure of maximal monotone graphs contained in 
\cite[Thm.~18, p.~126]{brezis-monot} and to infer that 
\[
  \xi \in \beta(u) \qquad\text{a.e.~in } \Omega\times(0,T)\times D\,.
\]
Finally, using the convergences of $(u^\eps)_\eps$ and $(\xi^\eps)_\eps$ it is not difficult to 
pass to the limit as $\eps\searrow0$ is the integral equation satisfies by $(u^\eps,\xi^\eps)$
and conclude that $(u,\xi)$ is thus a strong solution to the original problem.
The proof of Theorem~\ref{thm1} is concluded.


\section{Well-posedness with multiplicative noise}
\setcounter{equation}{0}
\label{mult}

In this section, we present the proof of Theorems~\ref{thm3}--\ref{thm4}.
The existence of a strong solution in the case of multiplicative noise is shown using a fixed point argument on 
subintervals of $[0,T]$ and through a classical iterative technique, while 
the continuous dependence follows from a natural procedure using the Lipschitzianity of $B$.

\subsection{Proof of existence}
Let $u_0$, $g$ and $B$ satisfy the assumptions \eqref{g}--\eqref{u0} and \eqref{mult1}--\eqref{mult3} of Theorem~\ref{thm3}.

For any progressively measurable process 
$v\in L^2(\Omega; L^2(0,T; V_1))$, let us show that the operator $B(\cdot,\cdot, v)$ and 
the initial datum $u_0$ satisfy the assumptions
of Theorem~\ref{thm1}. Indeed, since $B$ takes values in $\cL_2(U,H_0)$, 
it is clear that $(B(\cdot,\cdot,v)\cdot W)_D=0$, so that 
the integrability condition \eqref{add2} is easily satisfied. Moreover, by \eqref{mult3}, we have that
\[
  \norm{B(\cdot,\cdot,v)}_{\cL_2(U,H)}\leq |f| + C_B\norm{v}_{V_1} \in L^2(\Omega\times(0,T))\,,
\]
so that also \eqref{add1} is holds true. Hence, Theorems~\ref{thm1}--\ref{thm2} ensure the existence of a
strong solution $(u,\xi)$ such that
\[
  u(t) + \int_0^t\Delta^2 u(s)\,ds - \int_0^t\Delta\left(\xi(s) + \pi(u(s) + g(s))\right)\,ds
  =u_0 + \int_0^t B(s,v(s))\,dW_s
\]
for every $t\in[0,T]$, $\P$-almost surely, where the processes $u$ and $\xi-\xi_D$
are uniquely determined. We deduce that it is well-defined the map
\[
  \Lambda: L^2(\Omega; L^2(0,T; V_1)) \to L^2(\Omega; C^0([0,T]; V_1^*))\cap L^2(\Omega; L^\infty(0,T; H))\cap 
  L^2(\Omega; L^2(0,T; V_2))\,,
\]
where $\Lambda v$ is defined as the unique process $u$ solving the equation above 
(with a certain $\xi$).

We show that $\Lambda$ admits a unique fixed point. To this end, let $v_i\in L^2(\Omega; L^2(0,T; V_1))$
progressively measurable and $u_i:=\Lambda v_i$ solving the equation above for certain
$\xi_i$, for $i=1,2$: the continuous dependence property 
contained in Theorem~\ref{thm2} and the Lipschitz continuity of $B$ implies that 
\[\begin{split}
  \norm{u_1-u_2}_{L^2(\Omega; C^0([0,T]; V_1^*))} &+ \norm{\nabla(u_1-u_2)}_{L^2(\Omega; L^2(0,T; H))}\\
  &\lesssim\norm{B(v_1)-B(v_2)}_{L^2(\Omega; L^2(0,T; \cL_2(U,V_1^*)))}\\
  &\leq C_B\norm{v_1-v_2}_{L^2(\Omega; L^2(0,T; V_1^*))}\,.
\end{split}\]
Since $(u_1-u_2)_D=0$, this ensures that $\Lambda$ can be uniquely extended by continuity to
\[
  \Lambda': L^2(\Omega; L^2(0,T; V_1^*)) \to L^2(\Omega; C^0([0,T]; V_1^*))\cap L^2(\Omega; L^2(0,T; V_1))
\]
and
\[
  \norm{u_1-u_2}_{L^2(\Omega; L^2(0,T; V_1^*))}\lesssim
  C_B\sqrt{T}\norm{v_1-v_2}_{L^2(\Omega; L^2(0,T; V_1^*))}\,.
\]
Consequently, choosing $T_0>0$ sufficiently small, we have that $\Lambda'$ is a contraction
on $L^2(\Omega; L^2(0,T_0;V_1^*))$, and admits hence a unique fixed point $u$, which satisfies $\Lambda' u=u$.
Since also $u=\Lambda'u \in L^2(\Omega; L^2(0,T; V_1))$ by definition of $\Lambda'$,
we have that $u=\Lambda u$ with a corresponding $\xi$. 
It is clear that $(u,\xi)$ is a strong solution to \eqref{eq1}--\eqref{eq4} on $(0,T_0)$. 
A strong solution on the whole interval $[0,T]$ is obtained
using a standard technique by patching together all the local solution on $[0,T_0]$, $[T_0, 2T_0]$, $\ldots$ until $T$.

\subsection{Proof of continuous dependence}
Let $(u_0^i, g_i)$ and $B$ satisfy the assumptions \eqref{g}--\eqref{u0} 
and \eqref{mult1}--\eqref{mult3} of Theorem~\ref{thm4} and
$(u_i,\xi_i)$ be any respective solutions, for $i=1,2$.
By the continuous dependence property in Theorem~\ref{thm2} and the Lipschitz continuity of $B$, it follows that 
\[\begin{split}
  &\norm{u_1-u_2}_{L^2(\Omega; C^0([0,T]; V_1^*))} + \norm{\nabla(u_1-u_2)}_{L^2(\Omega; L^2(0,T; H))}\\
  &\lesssim \norm{u_0^1 - u_0^2}_{L^2(\Omega; V_1^*)} +\norm{g_1-g_2}_{L^2(\Omega; L^2(0,T; V_1^*))}
  + \norm{B(u_1)-B(u_2)}_{L^2(\Omega; L^2(0,T; \cL_2(U,V_1^*)))}\\
  &\lesssim \norm{u_0^1 - u_0^2}_{L^2(\Omega; V_1^*)} +\norm{g_1-g_2}_{L^2(\Omega; L^2(0,T; V_1^*))}
  + C_B\norm{u_1-u_2}_{L^2(\Omega; L^2(0,T; V_1^*))}\,.
\end{split}\]
Since $T$ is arbitrary, the continuous dependence is implied by the Gronwall lemma,
while the uniqueness follows directly
as in the proof of Theorem~\ref{thm2}.


\section{Regularity}
\setcounter{equation}{0}
\label{reg}
This section is devoted to proving the regularity result contained in Theorem~\ref{thm5}.

We use exactly the same approximation on the problem as described
in Subsection~\ref{approximation}, but we show that with the additional assumptions the approximated solutions
satisfy a further estimate involving the Ginzburg-Landau functional of the system.
Using the superscript $\delta$ to denote the action of the operator $(I-\delta\Delta)^{-2}$, as usual, 
and setting $w_\lambda:=-\Delta u_\lambda + \beta_\lambda(u_\lambda) + \pi(u_\lambda) + g$, if we
apply $(I-\delta\Delta)^{-2}$ to 
\eqref{app_eq} we get
\[\begin{split}
  u_\lambda^\delta(t) - \int_0^t\Delta w_\lambda^\delta(s)\,ds 
  =u_0^\delta + \int_0^tB^\delta(s)\,dW_s \qquad\text{in } V_1\,, \quad\forall\,t\in[0,T]\,, \quad\P\text{-a.s.}
\end{split}\]
Note that thanks to the regularities of $u_\lambda$ and $w_\lambda$, 
all the integrands in the previous equation are well-defined in the space $V_1$. Hence, we can apply the classical It\^o's
formula to the function
\[
  F_\lambda:V_1\to[0,+\infty)\,, \qquad F_\lambda(x):=\frac12\norm{\nabla x}_H^2 + 
  \int_Dj_\lambda(x) + \int_D\widehat\pi(x)\,, \quad x\in V_1\,,
\]
where $\widehat\pi(r):=\int_0^r\pi(s)\,ds$, $r\in\erre$.
We shall assume with no restriction from now on that $\beta_\lambda, \pi\in C^2(\erre)$ and have 
bounded first and second order derivatives: this is typically not the case,
because $\beta_\lambda$ and $\pi$ are just Lipschitz continuous on $\erre$.
To be precise, one should introduce a further approximation on $\beta_\lambda$ and $\pi$
in order to recover such regularity: for example, we could work with $\beta_{\lambda, i}$ and $\pi_i$,
where the subscript $i\in\enne$ denotes the convolution with a suitable mollifier $\rho_i:\erre\to\erre$.
This is absolutely not restrictive:
however, in order to avoid heavy notations and to make the treatment lighter, 
we shall simply assume such regularity on $\beta_\lambda$ and $\pi$ in the sequel.
With this further regularity, using the fact that $V_1\embed L^4(D)$, an easy computation shows that
$F_\lambda$ is twice Fr\'echet differentiable with derivatives given by
\begin{gather*}
  DF_\lambda(x)[y]=\sp{-\Delta x + \beta_\lambda(x) + \pi(x)}y_{V_1}=
  \int_D\nabla x\cdot\nabla y + \int_D(\beta_\lambda(x)+\pi(x))y\,,\\
  D_2F_\lambda(x)[y_1,y_2]=
  \int_D\nabla y_1\cdot\nabla y_2 + \int_D\left(\beta_\lambda'(x)+\pi'(x)\right)y_1y_2\,,
\end{gather*}
 for every $x,y, y_1,y_2\in V_1$.
As a consequence, if we set 
$\widetilde w_\lambda^\delta:=-\Delta u_\lambda^\delta + \beta_\lambda(u_\lambda^\delta)+\pi(u_\lambda^\delta)+ g$,
we deduce that
\[
  DF_\lambda(u_\lambda^\delta)[y]=\sp{-\Delta u_\lambda^\delta + \beta_\lambda(u_\lambda^\delta)+\pi(u_\lambda^\delta)}y_{V_1}
  =\sp{\widetilde w_\lambda^\delta-g}y_{V_1}=\int_D(\widetilde w_\lambda^\delta-g)y \quad\forall\,y\in V_1\,.
\]
Bearing in mind these considerations and recalling that 
$g$ is $V_1$-valued by assumption, 
It\^o's formula for $F_\lambda(u_\lambda^\delta)$ (see \cite[Thm.~4.32]{dapratozab})
yields
\[\begin{split}
  \frac12&\norm{\nabla u_\lambda^\delta(t)}_H^2 + \int_Dj_\lambda(u_\lambda^\delta(t)) + \int_D\widehat\pi(u_\lambda^\delta(t))
  + \int_0^t\!\!\int_D\nabla w_\lambda^\delta(s)\cdot\nabla\left(\widetilde w_\lambda^\delta(s)-g(s)\right)\,ds\\
  &=\frac12\norm{\nabla u_0^\delta}_H^2 + \int_Dj_\lambda(u_0^\delta) + \int_D\widehat\pi(u_0^\delta)\\
  &+\frac12\int_0^t\operatorname{Tr}\left[(B^\delta(s))^*D_2F_\lambda(u_\lambda^\delta(s))B^\delta(s)\right]\,ds
  +\int_0^t\left(\widetilde w_\lambda^\delta(s)-g(s), B^\delta(s)\right)_H\,dW_s
\end{split}\]
for every $t\in[0,T]$, $\P$-almost surely, hence also by Young's inequality
\[\begin{split}
  \frac12&\norm{\nabla u_\lambda^\delta(t)}_H^2 + \int_Dj_\lambda(u_\lambda^\delta(t))
  + \int_0^t\norm{\nabla\widetilde w_\lambda^\delta(s)}_H^2\,ds\\
  &\lesssim\norm{\nabla u_0^\delta}_H^2 + \int_Dj_\lambda(u_0^\delta) + \int_D\widehat\pi(u_0^\delta)
  +\norm{\nabla(\widetilde w_\lambda^\delta-w_\lambda^\delta)}_{L^2(0,T; H)}^2
  +\norm{g}^2_{L^2(0,T; V_1)}\\
  &+\frac12\int_0^t\operatorname{Tr}\left[(B^\delta(s))^*D_2F_\lambda(u_\lambda^\delta(s))B^\delta(s)\right]\,ds
  +\int_0^t\left(\widetilde w_\lambda^\delta(s)-g(s), B^\delta(s)\right)_H\,dW_s\,.
\end{split}\]
 Let us estimate the terms on the right-hand side separately.
First of all, by the contraction properties of the resolvent $(I-\delta\Delta)^{-2}$, the abstract Jensen inequality
(see \cite{haase}) and the Lipschitz continuity of $\pi$, we have
\[\begin{split}
  \norm{\nabla u_0^\delta}_H^2 + \int_Dj_\lambda(u_0^\delta) + \int_D\widehat\pi(u_0^\delta)&\lesssim_{C_\pi, C_0}
  1+\norm{\nabla u_0}^2_H+\int_D(I-\delta\Delta)^{-2}j_\lambda(u_0) + \norm{u_0^\delta}_H^2\\
  &\leq 1+\norm{u_0}_{V_1}^2 + \norm{j(u_0)}_{L^1(D)}\,.
\end{split}\]
Secondly, 
if $\{e_k\}_{k\in\enne}$ is an orthonormal system of $U$, we have
\[\begin{split}
  \operatorname{Tr}&\left[(B^\delta)^*D_2F_\lambda(u_\lambda^\delta)B^\delta\right] = 
  \sum_{k=0}^\infty\int_D|\nabla B^\delta e_k|^2 + 
  \sum_{k=0}^\infty\int_D(\beta_\lambda'(u_\lambda^\delta)+\pi'(u_\lambda^\delta))|B^\delta e_k|^2\\
  &\leq\norm{B}^2_{\cL_2(U,V_1)} + C_\pi\norm{B}^2_{\cL_2(U,H)}
  +\sum_{k=0}^\infty\int_D\beta_\lambda'(u_\lambda^\delta)|B^\delta e_k|^2\,.
\end{split}\]
Moreover, 
using the fact that $B$ takes values in $\cL_2(U,H_0)$ and recalling the definition of $\mathcal N$ and $\norm{\cdot}_{*}$,
we estimate the last term using the Burkholder-Davis-Gundy and Young inequalities in the following way,
for any $\eps>0$:
\[\begin{split}
  &\E\sup_{t\in[0,T]}\abs{\int_0^t\left(\widetilde w_\lambda^\delta(s)-g(s), B^\delta(s)\right)_H\,dW_s}\\
  &\lesssim
  \E\sup_{t\in[0,T]}\abs{\int_0^t\left(\widetilde w_\lambda^\delta(s), -\Delta\mathcal NB^\delta(s)\right)_H\,dW_s}+
  \E\sup_{t\in[0,T]}\abs{\int_0^t\left(g(s), B^\delta(s)\right)_H\,dW_s}\\
  &\lesssim\E\left(\int_0^T\norm{\nabla \widetilde w_\lambda^\delta(s)}_H^2\norm{B^\delta(s)}^2_{\cL_2(U,V_1^*)}\,ds\right)^{1/2}
  +\E\left(\int_0^T\norm{g(s)}_{V_1}^2
  \norm{B^\delta(s)}^2_{\cL_2(U,V_1^*)}\,ds\right)^{1/2}\\
  &\lesssim\eps\E\norm{\nabla \widetilde w_\lambda^\delta}^2_{L^2(0,T; H)} + 
  \left(\frac1{4\eps}+1\right)\norm{B}^2_{L^2(\Omega; L^\infty(0,T; \cL_2(U,V_1^*)))}
  +\norm{g}^2_{L^2(\Omega\times(0,T); V_1)}\,.
\end{split}\]
Bearing in mind this computations, taking supremum in time and expectations in It\^o's formula
and choosing $\eps$ sufficiently small we deduce that, for every $t\in[0,T]$,
\[\begin{split}
  \E &\sup_{r\leq t}\norm{\nabla u_\lambda^\delta(r)}_H^2 + \E\sup_{r\leq t}\int_Dj_\lambda(u_\lambda^\delta(r))
  +\E\norm{\nabla\widetilde w_\lambda^\delta}_{L^2(0,t;H)}^2\\
  &\lesssim1+\norm{u_0}_{L^2(\Omega; V_1)}^2 + \norm{j(u_0)}_{L^1(\Omega\times D)}
  +\E\norm{\nabla(\widetilde w_\lambda^\delta-w_\lambda^\delta)}_{L^2(0,T; H)}^2
  +\norm{g}^2_{L^2(\Omega; L^2(0,T; V_1))}\\
  &\quad+
  \norm{B}^2_{L^2(\Omega; L^2(0,T; \cL_2(U,V_1))\cap L^\infty(0,T; \cL_2(U,V_1^*)))}
  +\E\int_0^t\sum_{k=0}^{\infty}\int_D\beta_\lambda'(u_\lambda^\delta(s))|B^\delta(s)e_k|^2\,.
\end{split}\]
Now, we handle the last term on the right-hand side under the three different
hypotheses \eqref{ip_j_pol}--\eqref{ip_j3}.
Let us assume first condition \eqref{ip_j_pol}, so that we have $B=B_1+B_2$ with 
$B_1\in L^\infty(\Omega\times(0,T); \cL_2(U,H))$
and $B_2 \in L^2(0,T; L^\infty(\Omega; \cL_2(U,V_1)))$. Taking into account that 
$V_2\embed L^\infty(D)$, we have
\[\begin{split}
  \E\int_0^T\sum_{k=0}^\infty\int_D\beta_\lambda'(u_\lambda^\delta)|B_1^\delta e_k|^2&\lesssim
  \E\int_0^T\left(\norm{B_1(s)}^2_{\cL_2(U,H)}+
  \norm{u_\lambda^\delta(s)}^2_{L^{\infty}(D)}\norm{B_1^\delta(s)}^2_{\cL_2(U,H)}\right)\,ds\\
  &\lesssim\left(1+\norm{u_\lambda}^2_{L^2(\Omega; L^2(0,T;V_2))}\right)\norm{B_1}^2_{ L^\infty(\Omega\times(0,T); \cL_2(U,H))}
  \leq M
\end{split}\]
by \eqref{est1}, while thanks to
the H\"older inequality and the fact that
$V_1\embed L^6(D)$
\[\begin{split}
  \E&\int_0^t\sum_{k=0}^\infty\int_D\beta_\lambda'(u_\lambda^\delta(s))|B_2^\delta(s) e_k|^2\,ds\\
  &\leq R\norm{B_2}^2_{L^2(\Omega\times(0,T);\cL_2(U,H))}+
  R\E\int_0^t\sum_{k=0}^\infty\norm{|u_\lambda^\delta(s)|^2}_{L^{3/2}(D)}\norm{|B_2^\delta(s) e_k|^2}_{L^3(D)}\,ds\\
  &\lesssim
  \norm{B_2^\delta}^2_{L^2(\Omega\times(0,T);\cL_2(U,H))} 
  + \E\int_0^t\norm{u_\lambda^\delta(s)}^2_{L^3(D)}\sum_{k=0}^\infty\norm{B_2^\delta(s) e_k}^2_{L^6(D)}\,ds\\
  &\lesssim \norm{B_2}^2_{L^2(\Omega\times(0,T);\cL_2(U,H))} 
   +	\int_0^t\norm{B_2(s)}^2_{L^\infty(\Omega; \cL_2(U,V_1))}\E\sup_{r\leq s}\norm{u_\lambda}^2_{V_1}\,ds\,.
\end{split}\]
Hence,
noting that $\norm{u_\lambda^\delta}_1^2=|(u_0)_D|^2 + \norm{\nabla u_\lambda^\delta}_{H}^2$
(because $B$ takes values in $\cL_2(U,H_0)$),
we infer that
\[\begin{split}
  &\E \sup_{r\leq t}\norm{\nabla u_\lambda^\delta(r)}_H^2 + \E\sup_{r\leq t}\int_Dj_\lambda(u_\lambda^\delta(r))
  +\E\norm{\nabla\widetilde w_\lambda^\delta}_{L^2(0,t;H)}^2\\
  &\lesssim1+\norm{u_0}_{L^2(\Omega; V_1)}^2 + \norm{j(u_0)}_{L^1(\Omega\times D)}
  +\E\norm{\nabla(\widetilde w_\lambda^\delta-w_\lambda^\delta)}_{L^2(0,T; H)}^2
  +\norm{g}^2_{L^2(\Omega; L^2(0,T; V_1))}\\
  &+\norm{B}^2_{L^2(\Omega; L^2(0,T; \cL_2(U,V_1))\cap L^\infty(0,T; \cL_2(U,V_1^*)))}
  +\int_0^t\norm{B_2(s)}^2_{L^\infty(\Omega; \cL_2(U,V_1))}\E\sup_{r\leq s}\norm{\nabla u_\lambda^\delta(r)}_H^2\,ds\,.
\end{split}\]
Otherwise, under the assumption \eqref{ip_j2} or \eqref{ip_j3},
we have the inclusion $V_s\embed L^\infty(D)$ for any $s>1$ or $s>\frac32$, respectively, 
so that by the Jensen inequality and Remark \ref{yosida}
we infer similarly
\[\begin{split}
  \E&\int_0^t\sum_{k=0}^\infty\int_D\beta_\lambda'(u_\lambda^\delta(s))|B^\delta(s) e_k|^2\,ds\leq
  \E\int_0^t\sum_{k=0}^\infty\norm{B^\delta(s) e_k}_{L^\infty(D)}^2\norm{\beta_\lambda'(u_\lambda^\delta(s))}_{L^1(D)}\,ds\\
  &\lesssim \norm{B}_{L^2(\Omega; L^2(0,T;\cL_2(U,V_s)))}^2
  +\int_0^t\norm{B(s)}^2_{L^\infty(\Omega; \cL_2(U,V_s))}\E\sup_{r\leq s}\norm{j_\lambda(u_\lambda^\delta(r))}_{L^1(D)}\,ds\,.
\end{split}\]
Taking these remarks into account, by the Gronwall lemma
we deduce that
\[\begin{split}
  &\E\norm{\nabla u_\lambda^\delta}_{L^\infty(0,T; H)}^2 
  + \E\norm{j_\lambda(u_\lambda^\delta)}_{L^\infty(0,T; L^1(D))}
  +\E\norm{\nabla\widetilde w_\lambda^\delta}^2_{L^2(0,T; H)}\\
  &\lesssim1+ \E\norm{u_0}_{V_1}^2 + \E\norm{j(u_0)}_{L^1(D)}
  + \norm{g}^2_{L^2(\Omega; L^2(0,T; V_1))}
  +\E\norm{\nabla(\widetilde w^\delta_\lambda-w^\delta_\lambda)}^2_{L^2(0,T; H)}
\end{split}\]
for every $\delta,\lambda\in(0,1)$,
where the implicit constant depends only on $C_\pi, C_0, N, B, s$ according to the different
hypotheses \eqref{ip_j_pol}, \eqref{ip_j2} or \eqref{ip_j3} that are in order.
Since
\[
  \widetilde w^\delta_\lambda-w^\delta_\lambda
  =\beta_\lambda(u_\lambda)^\delta - \beta_\lambda(u_\lambda^\delta)
  +\pi(u_\lambda)^\delta - \pi(u_\lambda^\delta)
  +g^\delta - g
\]
and $\beta_\lambda(u_\lambda), \pi(u_\lambda), g \in L^2(\Omega; L^2(0,T; V_1))$
for all $\lambda$, we have that
$\E\norm{\nabla(\widetilde w^\delta_\lambda-w^\delta_\lambda)}^2_{L^2(0,T; H)}$
is uniformly bounded in $\delta$ for any fixed $\lambda$.
We deduce that $u_\lambda\in L^2(\Omega; L^\infty(0,T; V_1))$ 
and $w_\lambda\in L^2(\Omega; L^2(0,T; V_1))$
for every $\lambda$.
Hence, $w_\lambda^\delta\to w_\lambda$
in $L^2(\Omega; L^2(0,T; V_1))$ as $\delta\searrow0$, and the $C^2_b$-regularity of $\beta_\lambda$ and $\pi$
ensures also that $\widetilde w_\lambda^\delta\to w_\lambda$ in $L^2(\Omega; L^2(0,T; V_1))$ as $\delta\searrow0$.
Consequently, 
letting $\delta\searrow0$, we infer that
\beq
  \label{est3}
  \norm{\nabla u_\lambda}^2_{L^2(\Omega; L^\infty(0,T; H))} + 
  \norm{j_\lambda(u_\lambda)}_{L^1(\Omega; L^\infty(0,T; L^1(D)))} + 
  \norm{\nabla w_\lambda}^2_{L^2(\Omega; L^2(0,T; H))}\leq M
\eeq
for every $\lambda\in(0,1]$ and a
positive constant $M$ independent of $\lambda$.

Completing now the proof of existence as in Section~\ref{additive} with this further information
yields the desired regularity result by the lower semicontinuity of the norms. 
Let us just give a sketch of the proof.
Setting $w:=-\Delta u + \xi + \pi(u) + g$,
from the definition of $w_\lambda$, condition \eqref{est3} and the relatively weak compactness of $\beta_\lambda(u_\lambda)$
in $L^1$, we deduce that
\begin{gather*}
  w\in L^1(\Omega; L^1(0,T; L^1(D)))\,, \quad \nabla w \in L^2(\Omega; L^2(0,T; H))\,,\\
  u\in L^2(\Omega; L^\infty(0,T; V_1))\,.
\end{gather*}
Since $u\in L^2(\Omega; L^\infty(0,T; V_1)\cap C^0([0,T]; V_1^*))$ we have that
$u\in L^2(\Omega; C^0([0,T];H))$.
Moreover, recalling that $\nabla w \in L^2(\Omega; L^2(0,T; H))$,
we have $-\Delta w \in L^2(\Omega; L^2(0,T; V_1^*))$, and
since $(-\Delta w)_D=0$,
we infer $w-w_D=\mathcal N(-\Delta w) \in L^2(\Omega; L^2(0,T; V_1))$, hence also $w \in L^2(\Omega; L^2(0,T; V_1))$.
It follows from the definition of $w$ that $\xi \in L^2(\Omega; L^2(0,T; H))$.
As far as the variational formulation is concerned, note that the second equation follows directly 
from the definition of $w$. Furthermore, since
\[
  \int_Du(t)\varphi - \int_0^t\!\!\int_Dw(s)\Delta\varphi\,ds= 
  \int_Du_0\varphi + \int_DB\cdot W(t)\varphi \qquad\forall\,\varphi\in \bar V_2\,,
\]
 we deduce that
\[
  \sp{\partial_t(u-B\cdot W)}\varphi_{V_1}=- \int_D\nabla w\cdot\nabla\varphi \qquad\forall\,\varphi\in V_1\,,
\]
from which 
\[
  u-B\cdot W \in L^2(\Omega; H^1(0,T; V_1^*))
\]
and the variational formulation of the first equation.


\def\cprime{$'$}

\end{document}